\documentclass[11pt]{amsart}

\usepackage{accents}
\usepackage{amsmath}
\usepackage{amsthm}
\usepackage{amsfonts}
\usepackage{bbm}
\usepackage{amssymb}
\usepackage{mathbbol}
\usepackage{tikz}
\usepackage{mathrsfs}

\usepackage{graphicx}
\usepackage{txfonts}
  \usepackage[figuresright]{rotating}
  \usepackage{floatpag} 
 \usepackage[all]{xy} 
 
 \usetikzlibrary{matrix,arrows,decorations.markings}

\newcommand{\Z}{\mathbb{Z}}
\newcommand{\C}{\mathbb{C}}

\newcommand{\mf[1]}{\mathfrak{#1}}

\newcommand{\dom}{\mathscr{P}^+}

\newcommand{\G}{\mathcal{G}}
\newcommand{\B}{\mathcal{B}}
\newcommand{\cL}{\mathscr{L}}
\newcommand{\cH}{\mathcal{H}}

\newcommand{\cP}{\mathcal{P}}
\newcommand{\g}{\mathfrak{g}}
\newcommand{\p}{\mathfrak{p}}
\newcommand{\fb}{\mathfrak{b}}
\newcommand{\h}{\mathfrak{h}}
\newcommand{\fu}{\mathfrak{u}}

\newcommand{\fn}{\mathfrak{n}}

\newcommand{\cU}{\mathcal{U}}

\newcommand{\cO}{\mathscr{O}}
\newcommand{\cI}{\mathscr{I}}
\newcommand{\X}{\mathcal{X}}
\DeclareMathOperator{\ad}{ad}

\DeclareMathOperator{\Hom}{Hom}
\DeclareMathOperator{\Ker}{Ker}
\DeclareMathOperator{\Ann}{Ann}
\DeclareMathOperator{\Image}{Image}
\newtheorem{Lem}{Lemma}[section]
\newtheorem{Thm}[Lem]{Theorem}
\newtheorem{Def}[Lem]{Definition}
\newtheorem{Cor}[Lem]{Corollary}
\newtheorem{Ex}[Lem]{Example}

\newtheorem{Rmk}[Lem]{Remark}
\newtheorem{Prop}[Lem]{Proposition}

\title[Root components for tensor product]{Root components for tensor product of affine Kac-Moody Lie algebra modules} 
\author{Samuel Jeralds}
\author{Shrawan Kumar}

\address{S. Jeralds and S. Kumar: Department of Mathematics, University of North Carolina, Chapel Hill, NC 27599-3250, USA}
\email{sjj280@live.unc.edu ; shrawan@email.unc.edu} 
\begin{document}

\maketitle

\section{Abstract}  Let $\mathfrak{g}$ be an affine Kac-Moody Lie algebra and let $\lambda, \mu$ be two dominant integral weights for $\mathfrak{g}$. We prove that under some mild restriction, for any positive root $\beta$,  $V(\lambda)\otimes V(\mu)$ contains $V(\lambda+\mu-\beta)$ as a 
component, where $V(\lambda)$ denotes the integrable highest weight (irreducible) $\mathfrak{g}$-module with highest weight $\lambda$. This extends the corresponding result by Kumar from the case of finite dimensional semisimple Lie algebras to the affine Kac-Moody Lie algebras. One crucial ingredient in the proof is the action of Virasoro algebra via the Goddard-Kent-Olive construction on the tensor product  $V(\lambda)\otimes V(\mu)$. Then, we prove the corresponding geometric results including the higher cohomology vanishing on the $\G$-Schubert varieties in the product partial flag variety $\G/\cP\times \G/\cP$ with coefficients in certain sheaves 
coming from the ideal sheaves of $\G$-sub Schubert varieties. This allows us to prove the surjectivity of the Gaussian map.

\section{Introduction}

Let $\mf[g]$ be a symmetrizable Kac--Moody Lie algebra, and fix two dominant integral weights $\lambda$, $\mu \in \dom$. To these, we can associate the integrable, highest weight (irreducible) representations $V(\lambda)$ and $V(\mu)$. Then, the content of the tensor decomposition problem is to express the product $V(\lambda) \otimes V(\mu)$ as a direct sum of irreducible components; that is, to find the decomposition 
$$
V(\lambda) \otimes V(\mu) = \bigoplus_{\nu \in \dom} V(\nu)^{\oplus m_{\lambda, \mu}^\nu},
$$
where $m_{\lambda, \mu}^\nu \in \Z_{\geq 0}$ is the multiplicity of $V(\nu)$ in $V(\lambda) \otimes V(\mu)$. While this is a classical problem with a straightforward statement, determining the multiplicities $m_{\lambda, \mu}^\nu$ exactly--or even determining when $m_{\lambda, \mu}^\nu > 0$--is a challenging endeavor. Various algebraic, geometric, and combinatorial methods have been developed to understand the tensor decomposition problem; see \cite{Ku4} for a survey in the case of (finite dimensional) semisimple Lie algebras. 

While having a complete description for the components of a tensor product is desirable, many significant results in the literature demonstrate the existence of "families" of components; that is, components that are uniformly described and exist for tensor product decompositions regardless of $\mf[g]$. One such example is given by the {\it  root components} $V(\lambda + \mu-\beta)$ for a positive root $\beta$.

We show the existence of root components for affine Lie algebras generalizing the corresponding result in the finite case (i.e., when $\mf[g]$ is a semisimple Lie algebra) as in \cite{Ku1}. Recall that by a {\it Wahl triple} (introduced in \cite{Ku1} though not christened as {\it Wahl triple} there) we mean a triple $(\lambda, \mu, \beta) \in ({\dom})^2 \times \Phi^+$ such that 
\begin{itemize}
\item[(P1)] $\lambda+\mu-\beta \in \dom$, and 
\item[(P2)] If $\lambda(\alpha_i^\vee)=0$ or $\mu(\alpha_i^\vee)=0$, then $\beta- \alpha_i \not \in \Phi \sqcup \{0\}$,
\end{itemize} 
where  $\Phi$ (resp. $\Phi^+$) is the set of all the roots (resp., positive roots).

The main representation theoretic result of the paper is the following theorem (cf. Theorem \ref{mainthm}).

\vskip1ex

\noindent
{\bf Theorem I.} {\it For any affine Kac-Moody Lie algebra $\mf[g]$ and Wahl triple $(\lambda, \mu, \beta) \in ({\dom})^2 \times \Phi^+$, 
$$V(\lambda+\mu-\beta) \subset V(\lambda) \otimes V(\mu).$$}

We construct the proof in three parts. First, notice that the conditions (P1) and (P2) are invariant under adding $\delta$; that is, if $(\lambda, \mu, \beta)$ is a Wahl triple, then so is $(\lambda, \mu, \beta +k\delta)$ for any $k \in \Z_{\geq 0}$. This allows us to make use of the Goddard-Kent-Olive  construction of the Virasoro algebra action on 
the tensor product  $V(\lambda) \otimes V(\mu)$ and we
 explore its action on the subspaces $W^{\lambda+\mu-\beta}$ (as in Definition \ref{W}). This
reduces the problem  to certain `maximal root components' via Proposition \ref{L0}. In the second part, closely following the construction of root components for simple Lie algebras as in \cite{Ku1}, we show the existence of the bulk of the maximal root components.  Finally, in the third part, we construct the remaining maximal root components that are excluded from the previous methods explicitly using familiar, but ad hoc, constructions from the general tensor decomposition problem using the PRV components. 

We next prove the corresponding geometric results. Let $\G$ be the  `maximal' Kac-Moody group associated to the symmetrizable Kac-Moody Lie algebra $\g$ and let $\cP$ be a standard parabolic subgroup of $\G$ corresponding to a subset $S$ of the set of simple roots $\{\alpha_1, \cdots,\alpha_\ell\}$, i.e., $S$ is the set of simple roots of the Levi group of $\cP$. In the sequel, we abbreviate $S$ as the subset of $\{1, \cdots, \ell\}$. In particular, for $S=\emptyset$, we have the standard Borel subgroup $\B$ (corresponding to the  Borel subalgebra $\fb$).
Let $W$ be the Weyl group of $\g$ and let $W_\cP'$ be the set of smallest length coset representatives in $W/W_\cP$, where $W_\cP$ is the subgroup of $W$ generated by the simple reflections $\{s_k\}_{k\in S}$. 
For any pro-algebraic $\cP$-module $M$, by $\cL(M)$ we 
mean the corresponding homogeneous vector bundle on $\X_\cP:=\G/\cP$ associated to the 
principal $\cP$-bundle: $\G \to \G/\cP$ by the representation $M$ of $\cP$. 

For any integral weight $\lambda \in \h^*$ (where $\h$ is the Cartan subalgebra of $\g$ with the corresponding standard maximal torus $\cH$ of $\G$), such that $\lambda(\alpha_k^\vee)= 0$ for all $k \in S$, the one 
dimensional $\cH$-module $\mathbb{C}_\lambda$  (given by the character $\lambda$) admits a unique $\cP$-module 
structure (extending the $\cH$-module structure); in particular, we have the 
line bundle $\cL(\C_\lambda)$ on $\X_\cP$. We abbreviate the line bundle $\cL(\C_{-\lambda})$
 by $\cL(\lambda)$ and its restriction to the Schubert variety $X_w^\cP :=\overline{\B w\cP/\cP}$ by $\cL_w(\lambda)$ (for any $w\in W_\cP'$). Given two line bundles  $\cL(\lambda)$ and   $\cL(\mu)$, we can form their 
external tensor product to get the line bundle $\cL(\lambda\boxtimes \mu)$ on $\X_\cP\times \X_\cP$.
A dominant integral weight $\mu$ is called {\it $S$-regular} if $\mu(\alpha_k^\vee) = 0$ if and only if $k\in S$. The set of such weights is denoted by ${\dom_S}^o$.

Then, we prove that, for any $\mu\in {\dom_S}^o$ and $w\in W_\cP'$ such that $X_w^\cP$ is $\cP$-stable under the left multiplication:
$$H^p(X_w^\cP, \cI_e^k \otimes \cL_w(\mu)) = H^p(X_w^\cP, (\cO_w/\cI_e^k) \otimes \cL_w(\mu)) =
0, \,\,\text{for all $p > 0, k=1,2$},$$
where $\cI_e$ is the ideal sheaf of $X_w^\cP$ at the base point $e$ and $\mathscr{O}_w$  denotes the structure sheaf of $X_w^\cP$. We further explicitly determine $H^0(X_w^\cP, \cI_e^2 \otimes \cL_w(\mu))$ and $H^0(X_w^\cP, (\mathscr{O}_w/\cI_e^2) \otimes \cL_w(\mu))$ (cf. Proposition \ref{prop11.1} and Corollary \ref{remark11.1}).

For any $w\in W_\cP'$ such that  the Schubert variety $X_w^\cP$ is $\cP$-stable, define the $\G$-Schubert variety:
      $$\hat{\X}_w^\cP:= \G\times^{\cP} X_w^\cP.$$ 
      Consider the isomorphism:
      $$\delta:\G\times^\cP \X_\cP \simeq \X_\cP \times \X_\cP,\,\,\,[g, x]\mapsto (g\cP, gx),\,\,\text{for $g\in \G$ and $x\in \X_\cP$}.$$
    We have the canonical  embedding   
       $$\hat{\X}_w^\cP \hookrightarrow    \G\times^\cP \X_\cP .$$  
       In particular, we can restrict the line bundle $\cL(\lambda\boxtimes \mu)$ via the above isomorphism $\delta$ to 
      $\hat{\X}_w^\cP$ to get the line bundle denoted   $\cL_w(\lambda\boxtimes \mu)$.   Then, for $\lambda, \mu \in  {\dom_S}^o, k=1,2, w\in W_\cP'$ as above and $p\geq 0$,      
      we determine 
      (cf. Lemma \ref{lem12.1}):
  $$H^p(\hat{\X}_w^\cP, \hat{\cI}^k_e\otimes \cL_w(\lambda\boxtimes \mu))\,\,\text{and}\,\, H^p(\hat{\X}_w^\cP, (\hat{\cO}_w/\hat{\cI}^k_e)\otimes \cL_w(\lambda\boxtimes \mu)),$$
  in terms of the cohomology of the partial flag variety $\G/\cP$ with coefficients in explicit homogeneous vector bundles,
  where $\hat{\cI}_e$ denotes the ideal sheaf of $\widehat{\X}_e^\cP$ in  $ \hat{\X}_w^\cP$
and  $\hat{\cO}_w$ is the structure sheaf of $\hat{\X}_w^\cP$. In fact, we show that 
$$H^p(\hat{\X}_w^\cP, \hat{\cI}_e\otimes \cL_w(\lambda\boxtimes \mu))  =0,\,\,\text{for all $p>0$}. $$
 Let $\Phi^+\subset \h^*$ be the set of positive roots, $\Phi^+_S :=
\Phi^+\cap (\oplus_{k\in S}\,\mathbb{Z}_{\geq 0}\alpha_k)$ and  $\Phi^+(S):= \Phi^+ \setminus \Phi^+_S.$

We next study the vanishing of the first cohomology $H^1(\hat{\X}_w^\cP, \hat{\cI}_e^2\otimes \cL_w(\lambda\boxtimes \mu))$ and prove the following crucial result (cf. Proposition \ref{prop12.1}):
 
\vskip1ex

\noindent
{\bf Proposition I.} {\it  Let $\g$ be an affine Kac-Moody Lie algebra. Then, for any   $\lambda, \mu \in  {\dom_S}^o $ (where $S$ is an arbitrary subset of the simple roots of $\g$) and  any $w\in W_\cP'$ such that  the Schubert variety $X_w^\cP$ is $\cP$-stable, consider the following two conditions:
  \vskip1ex
  
    (a) $H^1(\hat{\X}_w^\cP, \hat{\cI}_e^2\otimes \cL_w(\lambda\boxtimes \mu)) = 0.$
    \vskip1ex
    
    (b) For all the real roots $\beta \in \Phi^+(S)$, satisfying $S\subset \{0\leq i \leq \ell: \beta-\alpha_i \not \in \Phi^+ \sqcup \{0\}\}$ and $\lambda +\mu -\beta \in \dom$, there exists a 
        $f_\beta \in \Hom_\fb\left(\C_{\lambda +\mu -\beta}\otimes V(\lambda)^\vee, V(\mu)\right)$ such that 
        $$X_\beta(f_\beta(\C_{\lambda +\mu -\beta}\otimes  v_\lambda^*))\neq 0,\,\,\text{for $X_\beta\neq 0 \in \g_\beta$},$$
        where $V(\lambda)^\vee$ is the restricted dual of $V(\lambda)$,
 and $v_\lambda^* \neq 0\in [V(\lambda)^\vee]_{-\lambda}$.
 \vskip1ex
 
        Then, the condition (b) implies the condition (a). }

 Further, we show that under the assumptions of the above proposition, the condition (b) of the proposition is satisfied in all the cases except possibly $\mathring{\g}$ of type $F_4$ or $G_2$ (cf. Proposition \ref{prop12.2} for a more precise result). The proof of Proposition \ref{prop12.2} relies on explicit constructions of root components obtained in the earlier sections including that of the GKO operator and the following lemma (cf. Lemma \ref{lem12.5}):
 
 \vskip1ex
 \noindent
 {\bf Lemma II.}  {\it Let $(\lambda, \mu, \beta)$ be a Wahl triple for a real root $\beta$ and let $V(\lambda+\mu-\beta) \subset V(\lambda) \otimes V(\mu)$ be a $\delta$-maximal root component. Observe that $\beta \in \mathring{\Phi}^+$ or $\beta=\delta-\gamma$ for $\gamma \in \mathring{\Phi}^+$. Then, 
 for $\beta \in \mathring{\Phi}^+$, the validity of condition (b) of Proposition \ref{prop12.1} for $V(\lambda+\mu-\beta)$
 implies its validity for $V(\lambda+\mu-\beta-k\delta)$ for any $k \geq 0$. 

Moreover, if $\beta=\delta-\gamma$ for $\gamma \in \mathring{\Phi}^+$, then we have an identity connecting the  condition (b) of Proposition \ref{prop12.1} for $V(\lambda+\mu-\delta+\gamma)$ with that of 
$V(\lambda+\mu-\beta)$
(see the identity \eqref{eqn12.6.10}).}
 
 \vskip1ex
 
 Combining the above Proposition I  and Proposition \ref{prop12.2}, we obtain  the following main geometric result of the paper (cf. Thoerem \ref{thm12.1}):
 \vskip1ex

\noindent
{\bf Theorem III.} {\it   Let $\g$ be an affine Kac-Moody Lie algebra and let  $w\in W_\cP'$ be such that  the Schubert variety $X_w^\cP$ is $\cP$-stable. Then, for any   $\lambda, \mu \in  {\dom_S}^o $ (where $S$ is an arbitrary subset of the simple roots) such that the condition (b) of Proposition  \ref{prop12.1} is satisfied for all the Wahl triples $(\lambda, \mu, \beta)$ for any real root $\beta \in \Phi^+$,
   \begin{equation*}  H^p(\hat{\X}_w^\cP, \hat{\cI}_e^2\otimes \cL_w(\lambda\boxtimes \mu)) = 0,\,\,\,\text{for all $p>0$}.
   \end{equation*}
   In particular, the canonical Gaussian map 
    \begin{equation*}  
   H^0(\hat{\X}_w^\cP, \hat{\cI}_e\otimes \cL_w(\lambda\boxtimes \mu))   \to  H^0(\hat{\X}_w^\cP, (\hat{\cI}_e/\hat{\cI}_e^2)\otimes \cL_w(\lambda\boxtimes \mu)) \end{equation*}
   is surjective.  
   
   In particular,  the theorem holds for any simply-laced $\mathring{\g}$ and $\mathring{\g}$ of types $B_\ell, C_\ell$. Moreover, it also holds for   $\mathring{\g}$ of type $F_4$ in the case $\cP$ is the Borel subgroup $\B$.}

\vskip1ex

     As a fairly straight forward corollary of the above theorem, taking inverse limits, we get the following (cf. Corollary \ref{coro12.7}):
     
     \vskip1ex
     \noindent
     {\bf Corollary IV.} {\it  Under the notation and assumptions of the above theorem, the canonical Gaussian map 
 $$H^0(\X_\cP\times \X_\cP, \cI_D\otimes \cL(\lambda\boxtimes \mu))  \to H^0(\X_\cP\times \X_\cP, (\cI_D/\tilde{\cI}_D^2)\otimes \cL(\lambda\boxtimes \mu))$$
 is surjective,  where $\cI_D$ is the ideal sheaf of the diagonal $D \subset \X_\cP\times \X_\cP$ and  $\tilde{\cI}_D^2$ is defined as $ \varprojlim_w\,\hat{\cI}_e(w)^2.$ }

\section{Affine Lie algebras (Notation and Preliminaries)}

In this section, we recall the definition of affine Kac-Moody Lie algebras $\mf[g]$ 
and their root and weight lattices. For a more extensive treatment of $\mf[g]$ and its properties, see Chapters 6 and 7 of \cite{Kac}. 

Let $\mathring{\mf[g]}$ be a finite-dimensional simple Lie algebra over $\mathbb{C}$ with a fixed Borel subalgebra $\mathring{\mf[b]}$ and Cartan subalgebra $\mathring{\mf[h]}\subset \mathring{\mf[b]}$. We denote the rank of $\mathring{\mf[g]}$ (which is, by definition, the dimension of $\mathring{\mf[h]}$)  by $\ell$.
Let $\mathring{\Phi}\subset \mathring{\mf[h]}^*$ be the set 
of roots and $\mathring{\Phi}^+$ (resp. $\mathring{\Phi}^-$) be the subset of positive (resp. negative) roots. Then, the associated affine Kac-Moody Lie algebra is given, as a vector space, by 
$$
\mf[g] = \mathring{\mf[g]} \otimes \C[t, t^{-1}] \oplus \C K \oplus \C d,
$$
where $K$ is the central element and $d$ the derivation or scaling element. The Lie bracket in $\mf[g]$ is given by 
\begin{align} \label{bracket}
[x \otimes t^m+zK+\zeta d, x' \otimes t^n + z'K +\zeta' d] =& [x,x'] \otimes t^{m+n} + n\zeta x' \otimes t^n - m\zeta' x \otimes t^m \notag \\
&+ m \delta_{n, -m} (x | x')_{\mathring{\mf[g]}} K,
\end{align}
for $x, x' \in \mathring{\mf[g]}$, $n, m \in \Z$, $z, z', \zeta, \zeta' \in \C$, where $\delta_{m,-n}$ is the Kronecker delta and $( \cdot | \cdot)_{\mathring{\mf[g]}}$ is the invariant form on $\mathring{\mf[g]}$ normalized so that $(\theta | \theta)_{\mathring{\mf[g]}} =2$, where $\theta \in \mathring{\Phi}$ is the highest root of $\mathring{\mf[g]}$. {\it Through the paper we will always take this normalized form on $\mathring{\mf[g]}$.}

The Cartan subalgebra of $\mf[g]$ is given by $\mf[h]:= \mathring{\mf[h]} \oplus (\C K + \C d)$. We thus have
$$
\mf[h]^\ast:= \mathring{\mf[h]}^\ast \oplus (\C \delta + \C \Lambda_0),
$$
where $\delta$ is defined by $\delta(d)=1$, $\delta|_{\mathring{\mf[h]} \oplus \C K} =0$ and $\Lambda_0(K)=1$, $\Lambda_0|_{\mathring{\mf[h]} \oplus \C d} =0$ and $\mathring{\mf[h]}^\ast $ is the dual of $\mathring{\mf[h]}$ taking $d$ and $K$ to zero. 
 Setting $\alpha_0:= \delta-\theta$,  $\alpha_0^\vee:= K-\theta^\vee$ for $\theta$ as above, and $\{\alpha_i\}_{i=1}^{\ell}$, $\{\alpha_i^\vee\}_{i=1}^{\ell}$ the set of simple roots and coroots of $\mathring{\mf[g]}$, we have that $\{\alpha_i\}_{i=0}^{\ell}$, $\{\alpha_i^\vee\}_{i=0}^{\ell}$ are the simple roots and coroots of $\mf[g]$. 

Let $\Phi$ be the set of roots of $\mf[g]$ and $\Phi^+$ (resp. $\Phi^-$) be the subset of positive (resp. negative) roots. Then,  we can separate $\Phi = \Phi_{Re} \sqcup \Phi_{Im}$, a disjoint union of real and imaginary roots. These are given precisely by 
$$
\begin{aligned}
\Phi_{Re} &= \{ \beta + k \delta: \beta \in \mathring{\Phi}, \ k \in \Z\}, \\ 
\Phi_{Im} &= \{k \delta: k \in \Z \backslash 0 \}.
\end{aligned}
$$

\noindent Finally, let $\{\Lambda_i\}_{i=0}^{\ell}$ be the set of fundamental weights of $\mf[g]$ defined by $\Lambda_i(\alpha_j^\vee)=\delta_{ij}$. Then, the dominant integral weights of $\mf[g]$ are given by 
$
\dom =\displaystyle \left(\bigoplus_{i=0}^{\ell} \Z_{\geq 0} \Lambda_i \right) \oplus \C \delta.
$
The integrable highest weight (irreducible) $\mf[g]$-modules are parameterized by $\dom$. For $\lambda \in \dom$, let $V(\lambda)$ be the corresponding integrable highest weight $\mf[g]$-module.

Fix dominant weights $\lambda$, $\mu \in \dom$, and consider the tensor product $V(\lambda) \otimes V(\mu)$
 of representations. Note that, as $\delta(\alpha_i^\vee)=0$ for all $i$, the representation $V(k\delta)$ is one-dimensional for all $k \in \mathbb{C}$, and $V(\lambda) \otimes V(k\delta) \cong V(\lambda + k\delta)$. {\it Therefore, throughout the remainder of this paper, we can assume without loss of generality that $\lambda(d)=\mu(d)=0$ up to an appropriate twist by $V(k\delta)$. }

\begin{Def}
{\rm For $\lambda \in \dom$,  the integer $\lambda(K)$ is called  the \textit{level} of $\lambda$. Further, for any $v \in V(\lambda)$ (not necessarily a highest weight vector), we have  $K.v=\lambda(K)v$,  since $K$ is a central element  of $\mf[g]$. So, we say that $V(\lambda)$ has {\it level} $\lambda(K)$. }
\end{Def} 

If $\lambda$ is of level $l$ and $\mu$ is of level $m$, then each irreducible component of $V(\lambda) \otimes V(\mu)$ is of level $l+m$. Indeed, the Cartan component $V(\lambda+\mu)$ clearly has level $l+m$, and if $V(\nu) \subset V(\lambda) \otimes V(\mu)$, then  necessarily $\lambda+\mu-\nu \in \oplus_{i=0}^\ell \Z_{\geq 0} \alpha_i$, and $\alpha_i(K)=0$ for all $i$.

We also recall the following definition taken from \cite{BrKu} of $\delta$-\textit{maximal components of the tensor product}. 

\begin{Def} \label{dmax} ({\bf $\delta$-Maximal Components})
{\rm A component $V(\nu) \subset V(\lambda) \otimes V(\mu)$ is called {\it $\delta$-maximal} if $V(\nu+k\delta) \not \subset V(\lambda) \otimes V(\mu)$ for any $k > 0$. }
\end{Def}

\noindent Let $\mf[g]':=[\mf[g],\mf[g]]$ be the derived subalgebra of $\mf[g]$. By the Lie bracket given in (\ref{bracket}), we have  $\mf[g]'= \mathring{\mf[g]} \otimes \C[t, t^{-1}] \oplus \C K$. In particular, the $\mf[g]'$ action on $V(\lambda)$ cannot detect the weight $\delta$. Nevertheless, it is known that the restriction of $V(\lambda)$ to a $\mf[g]'$-module  remains irreducible (cf. Lemma 2.1.4 of \cite{Ku3}). This allows us to make the following definition. 

\begin{Def} \label{W}
{\rm Let $V(\nu) \subset V(\lambda) \otimes V(\mu)$ be a $\delta$-maximal component. We denote by $W^\nu$ the subspace 
$$
W^\nu :=\sum_{k \geq 0} V(\nu-k\delta)^{\oplus m_k} \subset V(\lambda) \otimes V(\mu),
$$
where $m_k:= m_{\lambda, \mu}^{\nu-k\delta}$ is the multiplicity of $V(\nu-k\delta)$ in $V(\lambda) \otimes V(\mu)$. }
\end{Def}

That is, $W^\nu$ is the $\nu$-isotypic component of $V(\lambda) \otimes V(\mu)$ with respect to the $\mf[g]'$-action. In the next section, we make use of the Virasoro algebra to more closely examine the structure of such $W^\nu$.

\section{Virasoro algebra and  Goddard--Kent--Olive Construction}

In this section, we recall the basics of the Virasoro algebra and its representation theory. In particular, we give an overview of the Goddard--Kent--Olive (GKO) construction of the Virasoro algebra and its action on tensor products of affine Lie algebra representations. We follow the exposition of \cite{KRR}. 

\begin{Def} {\rm The {\it Virasoro algebra} $Vir$ is a Lie algebra over $\C$ with basis $\{c, L_k | k \in \Z\}$ with commutation relations 
\begin{equation} \label{Vir}
[L_k, L_j] = (k-j)L_{k+j} + \frac{1}{12}(k^3-k) \delta_{k, -j}c, \ [Vir, c]=0.
\end{equation}}
\end{Def}

We set $Vir_0:= \C L_0 \oplus \C c$, and  the dual space $Vir_0^\ast:= \C h \oplus \C z$, where $\{h, z\}$ is the basis dual to that of $\{L_0, c\}$. Similar to the case of Kac--Moody representations, we have the notion of  highest weight representations of $Vir$:

\begin{Def} 
{\rm A representation $V$ of $Vir$ is defined to be a {\it highest weight representation with highest weight} $\lambda \in Vir_0^\ast$ if there is a vector $v \in V$ such that 
$$
X.v=\lambda(X)v \ \forall X \in Vir_0, \ \ L_k.v=0 \ \forall k \geq 1, \ \ V= U\left( \bigoplus_{k <0} \C L_k \right).v.
$$}
\end{Def}

The structure of highest weight $Vir$ representations is in many ways parallel to that of Kac--Moody representations. Denote by $V_\mu$ the $\mu$-weight space of $V$. Then, by the defining relations (\ref{Vir}) for $Vir$, we have  $L_{-k}: V_\mu \to V_{\mu+kh}$. Further, if $V$ has highest weight $\lambda$, then for any $v \in V$ (not necessarily highest weight vector) by the equation \eqref{Vir} we again have  $c.v=\lambda(c)v$. We refer to the value $\lambda(c)$ as the \textit{central charge} of the representation $V$. 

\begin{Def}
{\rm A $Vir$ representation $V$ is called {\it unitarizable} if there is a positive-definite Hermitian form $\langle \cdot | \cdot \rangle$ on $V$ satisfying $\langle L_k.v | w\rangle = \langle v | L_{-k}.w\rangle$ and $\langle cv | w\rangle = \langle v | cw\rangle$ for all $k \in \Z$ and $v, w \in V$.}
\end{Def}

The existence of a positive-definite form on a highest weight representation $V$ allows us to investigate the weight spaces of $V$. More specifically, the following lemma from \cite{BrKu} determines when the weight spaces $\lambda+kh$ of an irreducible, highest weight  representation with highest weight $\lambda$ are nontrivial.

\begin{Lem} \label{BKLem}
Let $V$ be a unitarizable, highest weight (irreducible) representation of $Vir$ with highest weight $\lambda$. Then, we have the following: 
\begin{enumerate}
\item If $\lambda(L_0) \neq 0$, then $V_{\lambda+kh} \neq 0$ for any $k \in \Z_{\geq 0}$.
\item If $\lambda(L_0) =0$ and $\lambda(c) \neq 0$, then $V_{\lambda+kh} \neq 0$ for all $k \geq 2$, and $V_{\lambda+h}=0$.
\item If $\lambda(L_0)=\lambda(c)=0$, then $V$ is one-dimensional. 
\end{enumerate}
\end{Lem}

\begin{proof}
Let $v$ be the highest weight vector of $V$. Then, we have (for all $k >0$):
$$
0 \leq \langle L_{-k}v | L_{-k}v\rangle = \langle L_k L_{-k}v | v\rangle = (2k\lambda(L_0) + \frac{1}{12}(k^3-k)\lambda(c))\langle v | v\rangle.
$$
Thus, $\lambda(L_0)$ and $\lambda(c)$ both must  be nonnegative numbers. So, if $\lambda(L_0)\neq 0$, we have  $L_{-k}v \neq 0$ for any $k \geq 0$. If $\lambda(L_0)=0$, then we have $(k^3-k)\lambda(c) \geq 0$; if $\lambda(c) \neq 0$, this is nonzero for $k >1$ and zero for $k=1$, so that $L_{-1}v=0$ and $L_{-k}v \neq 0$ for $k >1$. Finally, if $\lambda(L_0)=\lambda(c)=0$, then we get that $L_{-k}v=0$ for all $k >0$, so $V$ must be one-dimensional. 
\end{proof}

\noindent
{\bf The Goddard--Kent--Olive construction:}
The similarity between representations of Kac--Moody algebras and the Virasoro algebra is not coincidental. A foundational result linking the two theories is given by the \textit{Sugawara construction}, which embeds $Vir \hookrightarrow \hat{U}(\mf[g])$ in a certain completion   $\hat{U}(\mf[g])$ of the enveloping algebra  ${U}(\mf[g])$
of an affine Lie algebra $\mf[g]$ such  that an integrable highest weight (irreducible)  representation $V(\lambda)$ of $\mf[g]$ becomes a unitarizable representation of $Vir$. We make use of a related construction, known as the {\it Goddard--Kent--Olive (for short  GKO)  construction}, to produce an action of $Vir$ on a tensor product $V(\lambda) \otimes V(\mu)$ of $\mf[g]$-modules. At its core, the GKO construction is "relative" and relies on the diagonal inclusion $\mathring{\mf[g]} \hookrightarrow \mathring{\mf[g]} \oplus \mathring{\mf[g]}$, but we will not give complete details of the construction here; see Lecture 10 of \cite{KRR} for a more in-depth treatment. The following proposition is the primary computational tool we need from this construction. 

\begin{Prop}(\cite{KRR}, Proposition 10.3) \label{GKO}
Let $\mf[g]$ be an affine Lie algebra and $\lambda, \mu \in \dom$ be weights with levels $l$, $m$ repsectively. Then, 

(1)  $V(\lambda) \otimes V(\mu)$ is a unitarizable $Vir$ representation with non-negative central charge
$$
(\dim \mathring{\mf[g]}) \left( \frac{l}{l+h^\vee}+\frac{m}{m+h^\vee}-\frac{l+m}{l+m+h^\vee} \right),
$$
where $h^\vee$ is the dual Coxeter number of $\mf[g]$ (\cite{Kac}, $\S$6.1).

\vskip1ex

(2) $L_0$ acts on $V(\lambda) \otimes V(\mu)$ by 
$$
\frac{1}{2} \left( \frac{(\lambda | \lambda+2\rho)}{l+h^\vee}+\frac{(\mu|\mu+2\rho)}{m+h^\vee}-\frac{\Omega}{l+m+h^\vee} \right),
$$
where $\Omega$ is the Casimir operator of $\mf[g]$ (\cite{Ku3}, $\S$1.5) and $( \cdot | \cdot)$ is the normalized form on $\mf[h]^\ast$ as in \cite{Ku3}, Lemma 13.1.8. 

\vskip1ex

(3) For all $k$, $[L_k, \mf[g]']=0$; i.e., the $L_k$ are intertwining operators for the representation of $\mf[g]'$ on $V(\lambda) \otimes V(\mu)$.
\end{Prop}

\begin{Rmk} \label{scalar}
{\rm As the Casimir $\Omega$ acts on the $\mf[g]$-module $V(\nu)$ by $(\nu | \nu+2\rho)$ (cf., \cite{Ku3},  2.1.16), we can easily compute the action of $L_0$ on any component $V(\nu)$ of $V(\lambda) \otimes V(\mu)$; it will act via a scalar depending only on $\lambda, \mu, \nu$.}
\end{Rmk}

Now, let $V(\nu) \subset V(\lambda) \otimes V(\mu)$ be a $\delta$-maximal component as in Definition \ref{dmax} and consider the subspace $W^\nu$ as in Definition \ref{W}. By Proposition \ref{GKO} (3), we can conclude the following immediate corollary.

\begin{Cor}
$W^\nu$ is a unitarizable $Vir$ representation. 
\end{Cor}

In general, $W^\nu$ is not an irreducible $Vir$ representation. However, by Remark \ref{scalar},   $L_0$ acts on each $V(\nu-k\delta)$ as a scalar, so that each of these summands corresponds to a single $Vir$ weight space. 

\begin{Lem} \label{kshift} 
For any $k\in \Z$, $L_{-k}: V(\nu) \to V(\nu-k\delta)$, corresponding to shifting the $Vir$ weight by $kh$, where $\nu$ is not necessarily a $\delta$-maximal weight. 
\end{Lem}

\begin{proof}
As all of $W^\nu$ has the same non-negative central charge, we need only compute the change in the weight with respect to $h$. By Proposition \ref{GKO} and Remark \ref{scalar}, we have that $L_0$ acts on $V(\nu-k\delta)$ via 
$$
\begin{aligned}
&\frac{1}{2} \left(\frac{(\lambda| \lambda+2\rho)}{l+h^\vee}+\frac{(\mu|\mu+2\rho)}{m+h^\vee}-\frac{(\nu-k\delta|\nu-k\delta+2\rho)}{l+m+h^\vee} \right) \\
=&\frac{1}{2} \left(\frac{(\lambda| \lambda+2\rho)}{l+h^\vee}+\frac{(\mu|\mu+2\rho)}{m+h^\vee}-\frac{(\nu|\nu+2\rho)-2k(\nu|\delta)-2k(\rho|\delta)}{l+m+h^\vee} \right)\\
=&\frac{1}{2} \left(\frac{(\lambda| \lambda+2\rho)}{l+h^\vee}+\frac{(\mu|\mu+2\rho)}{m+h^\vee}-\frac{(\nu|\nu+2\rho)-2k(l+m)-2k(h^\vee)}{l+m+h^\vee} \right) \\
=&\frac{1}{2} \left(\frac{(\lambda| \lambda+2\rho)}{l+h^\vee}+\frac{(\mu|\mu+2\rho)}{m+h^\vee}-\frac{(\nu|\nu+2\rho)}{l+m+h^\vee} \right)+k.
\end{aligned}
$$
Now, as $L_0$ acts on $V(\nu)$  by $\frac{1}{2} \left(\frac{(\lambda| \lambda+2\rho)}{l+h^\vee}+\frac{(\mu|\mu+2\rho)}{m+h^\vee}-\frac{(\nu|\nu+2\rho)}{l+m+h^\vee} \right)$ and we have that $L_{-k}$ adds $kh$ to a weight space, we get the desired result.  
\end{proof}

 Therefore, combining Lemmas \ref{kshift} and \ref{BKLem} we get the following key proposition, which allows us to determine if the multiplicities $m_k=m_{\lambda, \mu}^{\nu-k\delta}$ as in Definition \ref{W} are nonzero. 

\begin{Prop}  \label{L0} Let $\lambda, \mu \in \dom$ be of positive levels $l, m$, respectively. Let $V(\nu) \subset V(\lambda) \otimes V(\mu)$ be a $\delta$-maximal component, and consider the $Vir$ subrepresentation $W^\nu$. Then, we have the following:

\vskip1ex
(1) If $\frac{(\lambda| \lambda+2\rho)}{l+h^\vee}+\frac{(\mu|\mu+2\rho)}{m+h^\vee}-\frac{(\nu|\nu+2\rho)}{l+m+h^\vee}  \neq 0$, we have $m_k \geq 1$ for all $k \geq 0$. 
\vskip1ex

(2) Else, we have $m_1 =0$ and $m_k \geq 1$ for $k=0$, $k \geq 2$.
\end{Prop}

\begin{proof}
Choose a nonzero highest weight vector $v_\nu \in V(\nu) \subset W^\nu$. By Lemma \ref{kshift}, we have $L_k(v_\nu)=0$ for all $k >0$, as there are no components of the form $V(\nu+k\delta)$ in $W^\nu$. Thus, $v_\nu$ is a highest weight vector for the $Vir$ action, and generates a highest weight, unitarizable (irreducible) $Vir$-submodule of $W^\nu$. Applying Lemma \ref{BKLem}, we   can determine if $L_{-k}(v_\nu)$ is nonzero  by computing the $L_0$-action on $v_\nu$ (done in the proof of Lemma  \ref{kshift}). But, $L_{-k}(v_\nu) \in V(\nu-k\delta)$  by Lemma \ref{kshift} and if nonzero it is the highest weight vector of $V(\nu-k \delta)$ under the $U(\mf[g]')$ action, since $[L_{-k}, \mf[g]']=0$. Therefore,  Lemma \ref{BKLem}
 gives the result. (Observe that $c$ acts by a nonzero scalar on 
$V(\lambda) \otimes V(\mu)$ by Proposition \ref{GKO} (1).)
\end{proof}

We conclude this section with two examples; the first is an alternate approach to \cite{Kac}, Exercise 12.16, and the second is an interpretation of root components for affine Lie algebras with respect to imaginary roots. 

\begin{Ex}
{\rm Consider $\mf[g]=\widehat{\mf[sl]}_2$, and let $\lambda=\mu=\Lambda_0$. Then, of course $V(2\Lambda_0) \subset V(\Lambda_0) \otimes V(\Lambda_0)$ is a $\delta$-maximal component. By Proposition \ref{GKO} and Remark \ref{scalar}, the $L_0$ action on $V(2\Lambda_0)$ is given by the scalar:
$$
\frac{1}{2}\left( \frac{(\Lambda_0|\Lambda_0+2\rho)}{3}+ \frac{(\Lambda_0|\Lambda_0+2\rho)}{3}- \frac{(2\Lambda_0|2\Lambda_0+2\rho)}{4}\right).
$$
But, we have $(\Lambda_0|\Lambda_0)=(\Lambda_0|\rho)=0$ (cf., \cite{Kac}, $\S$6.2), so that $L_0$ acts by $0$ on $V(2\Lambda_0)$. Thus,  by Proposition \ref{L0}, we have $V(2\Lambda_0-\delta) \not \subset V(\Lambda_0)\otimes V(\Lambda_0)$. But, $V(2\Lambda_0 -k\delta) \subset V(\Lambda_0) \otimes V(\Lambda_0) $ for any $k >1$; this agrees with Exercise 12.16 of \cite{Kac}.}
\end{Ex}

\begin{Ex} \label{imag}
{\rm For any affine Lie algebra $\mf[g]$, let $\lambda=\mu=\rho$. Then, of course,  $V(2\rho) \subset V(\rho) \otimes V(\rho)$ is $\delta$-maximal. We then get,  since $\rho$ is of level $h^\vee$, that $L_0$ acts on $V(2\rho)$ via 
$$
\frac{1}{2} \left( \frac{2(\rho | \rho+2\rho)}{2h^\vee} - \frac{(2\rho | 2\rho + 2\rho)}{3h^\vee} \right) = \frac{(\rho | \rho)}{6h^\vee} > 0.
$$
Thus,  by Proposition \ref{L0}, we have $V(2\rho-k\delta) \subset V(\rho) \otimes V(\rho)$ for any $k \geq 0$. }
\end{Ex}
\section{Some general results on tensor product decomposition} 

The following  proposition is crucial to our work on root components in the case of affine Lie algebras. For the semisimple Lie algebras, it is due to Kostant \cite{Kos}. 
\begin{Prop}\label{mult}
Let $\mf[g]$ be a symmetrizable Kac--Moody algebra, and fix $\lambda, \mu, \nu \in \dom$. Then, the multiplicity of $V(\nu)$ in $V(\lambda) \otimes V(\mu)$ is given by:
$$
m_{\lambda, \mu}^\nu = \dim \{v \in V(\mu)_{\nu-\lambda}: e_i^{\lambda(\alpha_i^\vee)+1}.v=0 \ \text{for all simple roots } \alpha_i\},
$$
where $e_i$ is a root vector corresponding to the root $\alpha_i$. 
\end{Prop}
\begin{proof} Define an increasing  filtration $\{F_n(\lambda)\}_{n\geq 0}$ of $V(\lambda)$ by
$$F_n(\lambda) := \sum_{\beta\in Q^+: |\beta| \leq n}\, V(\lambda)_{\lambda-\beta} ,$$
where $Q^+ :=\oplus_i \Z_{\geq 0}\alpha_i$ , for any $\beta =\sum n_i\alpha_i$, we define $|\beta| := \sum n_i$
and $V(\lambda)_\gamma$ denotes the $\gamma$-weight space. 
Then, clearly, each $F_n(\lambda)$ is $\mf[b]$-stable. We also set
$$V(\lambda)^\vee := \oplus_{\beta\in Q^+}\, \left(V(\lambda)_{\lambda-\beta}\right)^*.$$
Let $\mathbb{C}_\nu$ denote the one-dimensional $\mf[b]$-module with the Cartan subalgebra $\mf[h]$ acting via the weight $\nu$ and of course the commutator $\mf[n] := [\mf[b], \mf[b]]$ acts trivially. 
Now, 
\begin{align} \label{eqnnew} \Hom_{\mf[g]}\left(V(\nu), V(\lambda)\otimes V(\mu)\right)&\simeq  \Hom_{\mf[b]}\left(\mathbb{C}_\nu, V(\lambda)\otimes V(\mu)\right)\notag\\
&\simeq \varinjlim_n\, \Hom_{\mf[b]}\left(\mathbb{C}_\nu, F_n(\lambda)\otimes V(\mu)\right)\notag\\
&\simeq \varinjlim_n\, \Hom_{\mf[b]}\left(F_n(\lambda)^*\otimes \mathbb{C}_\nu,  V(\mu)\right),\notag\\
&\,\,\,\,\,\,\,\,\,\text{since $F_n(\lambda)$ is finite-dimensional}\notag\\
&\simeq  \Hom_{\mf[b]}\left(V(\lambda)^\vee\otimes \mathbb{C}_\nu,  V(\mu)\right),
\end{align}
where the last isomorphism  is induced from the surjection $V(\lambda)^\vee \to F_n(\lambda)^*$ obtained from dualizing the embedding $F_n(\lambda) \hookrightarrow V(\lambda)$. It is indeed an isomorphism since 
any $\mf[b]$-module homomorphism $f: V(\lambda)^\vee\otimes \mathbb{C}_\nu \to  V(\mu)$ clearly descends to a 
homomorphism $\bar{f}:  F_N(\lambda)^*\otimes \mathbb{C}_\nu \to  V(\mu)$ for some large $N$ because the weights of $V(\lambda)^\vee\otimes \mathbb{C}_\nu$ lie in $-\lambda+\nu +Q^+$ whereas the weights of $V(\mu)$ lie in 
$\mu-Q^+$. 

From the definition of the integrable highest weight module (cf. \cite{Ku3}, Definition 2.1.5 and Corollary 2.2.6) (rather its analogue for the integrable lowest module $V(\lambda)^\vee$), we get that 
 $V(\lambda)^\vee$ is obtained from the Verma module for the negative Borel $M_{\mf[b]^-}(-\lambda)$ (with lowest weight $-\lambda$) by taking its quotient by  the submodule generated by $\{e_i^{\lambda(\alpha_i^\vee)+1}. v_{-\lambda}\}$, where $e_i$ runs over all the (positive) simple root vectors and $v_{-\lambda}$ is the lowest weight vector of  $M_{\mf[b]^-}(-\lambda)$. In particular, the map 
 $$U(\mf[n]) \to V(\lambda)^\vee,\,\,\, a \mapsto a\cdot v_{-\lambda},$$
 is surjective with kernel the left ideal generated by $\{e_i^{\lambda(\alpha_i^\vee)+1}\}_i$. 
 From this and the above isomorphism  \ref{eqnnew}, we get the proposition.
 \end{proof}

The above result also follows from the work of Kashiwara \cite{Kas1} \cite{Kas2} on the existence of  crystal base and global base for quantum group representations. 

We can deduce from Proposition \ref{mult} the following well-known fact, which we refer to as {\it additivity in the tensor decomposition}.  Following \cite{Ku1}, Corollary 1.5, we give a purely algebraic proof,  but remark that this can also be deduced from the Kac-Moody analogue of the Borel--Weil theorem.

\begin{Cor} \label{additive} Let $\lambda$, $\lambda'$, $\mu$, $\mu'$, $\nu \in \dom$. Then, $m_{\lambda, \mu}^\nu \leq m_{\lambda+\lambda', \mu+\mu'}^{\nu+\lambda'+\mu'}$.
\end{Cor}

\begin{proof}
Define a map $\xi: V(\mu) \to V(\mu+\mu')$ via $\xi(v) = \pi( v \otimes v_{\mu'})$, where $v_{\mu'}$ is a nonzero highest weight vector of $V(\mu')$ and $\pi: V(\mu) \otimes V(\mu') \to V(\mu+\mu')$ is the projection onto the Cartan component. For any nonzero $v \in V(\mu)$,  $v_\mu \otimes v_{\mu'} \in U(\mf[g]).(v \otimes v_{\mu'})$, so the map $\xi$ is clearly injective. Now,  by Proposition \ref{mult}, choose a vector $v \in V(\mu)_{\nu-\lambda}$ satisfying $e_i^{\lambda(\alpha_i^\vee)+1}.v=0$ for all the simple roots $\alpha_i$. Then, $\xi(v) \in V(\mu+\mu')_{\nu+\mu'-\lambda}$ and satisfies $e_i^{\lambda(\alpha_i^\vee)+1}.(\xi(v))=0$. Thus, by Proposition \ref{mult}, we get  $m_{\lambda, \mu}^\nu \leq m_{\lambda, \mu+\mu'}^{\nu+\mu'}$. Repeating the argument a second time for $\lambda$ and $\lambda'$ gives the result. 
\end{proof}

\section{Existence of root components for imaginary roots}

Now, let $\beta=k\delta$ be a positive imaginary root. By the regularity condition (P2) for root components as in the Introduction, if $(\lambda, \mu, \beta)$ is a Wahl triple, then necessarily $\lambda$ and $\mu$ are both regular dominant weights. Combining Example \ref{imag} and additivity in the tensor decomposition (Corollary \ref{additive}), we get the first example of root components via the following corollary. 

\begin{Cor} \label{ImRootComp}
Let $\beta = k\delta $ be a positive imaginary root, and $(\lambda, \mu, \beta)$ a Wahl triple. Then, $V(\lambda+\mu-\beta) \subset V(\lambda) \otimes V(\mu)$. 
\end{Cor}

\begin{Rmk} {\rm In fact, the stronger statement that $V(\lambda+\mu-k\delta) \subset V(\lambda) \otimes V(\mu)$ is true if \textbf{at least one} of $\lambda$ or $\mu$ is regular dominant follows immediately from the same approach as in Example \ref{imag}. }
\end{Rmk}

\section{Existence of  root components  for $\beta = \gamma+k\delta$, for $\gamma \in \mathring{\Phi}^+$}

Fix for the remainder of this section a simple untwisted affine Lie algebra $\mf[g]$ and a Wahl triple $(\lambda, \mu, \beta) \in ({\dom})^2 \times \Phi^+_{Re}$. Without loss of generality, assume that $\lambda(d)=\mu(d)=0$. Let $l$, $m$ be the levels of $\lambda$ and $\mu$, respectively. Then, by Proposition \ref{GKO} (2), we know precisely how $L_0 \in Vir$ acts on $V(\lambda)\otimes V(\mu)$ via the GKO construction. The goal of this section is to understand how this action interacts with potential root components $V(\lambda+\mu-\beta)$. As a first step, we consider the Cartan component $V(\lambda+\mu) \subset V(\lambda) \otimes V(\mu)$ in the following lemma. 

\begin{Lem} \label{nonneglemma} For $\lambda, \mu \in \dom$ with levels $l$, $m$, respectively, we have 
$$
\frac{1}{2} \left( \frac{(\lambda | \lambda + 2\rho)}{l+h^\vee}+\frac{(\mu | \mu+2\rho)}{m+h^\vee} - \frac{(\lambda+\mu | \lambda+\mu + 2\rho)}{l+m+h^\vee} \right)  \geq 0.
$$
\end{Lem}

\begin{proof}
Let $v:=v_\lambda \otimes v_\mu \in V(\lambda) \otimes V(\mu)$, where $v_\lambda$, $v_\mu$ are nonzero highest weight vectors in $V(\lambda)$, $V(\mu)$, respectively. Then, $v$ is the highest weight vector of the Cartan component $V(\lambda+\mu) \subset V(\lambda) \otimes V(\mu)$. Now, consider the Virasoro action $Vir. v$ coming from the GKO construction. As $V(\lambda+ \mu)$ is a $\delta$-maximal component of $  V(\lambda) \otimes V(\mu)$, by Lemma \ref{kshift},   $L_k.v=0$ for all $k \geq 1$. So, consider the highest weight, unitarizable (irreducible) Virasoro submodule of $V(\lambda) \otimes V(\mu)$
generated by $v$. By Proposition \ref{GKO}, we know that the action of $L_0$ on $v$ is given precisely by the scalar value above. Further, by the proof of Lemma \ref{BKLem}, we know that this value must be nonnegative.

\end{proof}

Now, we can obtain from this the following key proposition pertaining to Wahl triples. 

\begin{Prop} \label{positive}
Let $(\lambda, \mu, \beta) \in (\dom)^2 \times \Phi^+_{Re}$ be a Wahl triple. Then, 
$$
\frac{1}{2} \left( \frac{(\lambda | \lambda + 2\rho)}{l+h^\vee}+\frac{(\mu | \mu+2\rho)}{m+h^\vee} - \frac{(\lambda+\mu -\beta | \lambda+\mu -\beta + 2\rho)}{l+m+h^\vee} \right) > 0.
$$
\end{Prop}

\begin{proof}
We use the bilinearity of the form $(\cdot | \cdot)$ to expand the third term as 
$$
\frac{(\lambda+\mu -\beta | \lambda+\mu -\beta + 2\rho)}{l+m+h^\vee} = \frac{(\lambda+\mu | \lambda+\mu+2\rho)}{l+m+h^\vee} - \frac{2(\lambda+\mu | \beta) + 2(\rho | \beta) - (\beta | \beta)}{l+m+h^\vee}.
$$
In total, we therefore have  
$$
\begin{aligned} 
&\frac{1}{2} \left( \frac{(\lambda | \lambda + 2\rho)}{l+h^\vee}+\frac{(\mu | \mu+2\rho)}{m+h^\vee} - \frac{(\lambda+\mu -\beta | \lambda+\mu -\beta + 2\rho)}{l+m+h^\vee} \right) \\
&= \frac{1}{2} \left( \frac{(\lambda | \lambda + 2\rho)}{l+h^\vee}+\frac{(\mu | \mu+2\rho)}{m+h^\vee} - \frac{(\lambda+\mu | \lambda+\mu + 2\rho)}{l+m+h^\vee} \right) +\\
&\,\,\,\,\,\,\, \frac{1}{2} \left(\frac{2(\lambda+\mu | \beta) + 2(\rho | \beta) - (\beta | \beta)}{l+m+h^\vee} \right).
\end{aligned}
$$
By Lemma \ref{nonneglemma}, we know that the first term is nonnegative. Further, we know that $2(\rho | \beta) - (\beta | \beta) \geq 0$, as $1 \leq \rho(\beta^\vee) = \frac{2(\rho|\beta)}{(\beta|\beta)}$. Finally, by the property (P1) of the Wahl triple,  we have  $(\lambda+\mu | \beta) >0$ (since $\lambda+\mu-\beta\in \dom$, we get $(\lambda+\mu-\beta | \beta)\geq 0$ and hence  $(\lambda+\mu | \beta) >0$), so that in total the expression as a whole is strictly positive. 
\end{proof}

As a  consequence of Propositions \ref{positive} and \ref{L0}, 
 we get the following corollary. 

\begin{Cor} \label{redux}
Let $(\lambda, \mu, \beta)\in (\dom)^2 \times \Phi^+_{Re}$ be a Wahl triple and let
$V(\lambda+\mu-\beta) \subset V(\lambda) \otimes V(\mu)$ be a $\delta$-maximal root component. In particular, $\lambda, \mu$ have strictly positive levels. Then, for any $k \geq 0$, we have $V(\lambda+\mu-\beta-k\delta) \subset V(\lambda) \otimes V(\mu).$
\end{Cor}
\begin{proof} By Propositions \ref{positive}, \ref{GKO} and Remark \ref{scalar}, $L_0$ acts on $V(\lambda +\mu-\beta)$ via a scalar $>0$. Thus, by Proposition \ref{L0}, 
$$ V(\lambda +\mu-\beta -k \delta) \subset V(\lambda) \otimes V(\mu),\,\,\,\text{for all $k\geq 0$}.$$
This proves the corollary.
\end{proof} 

Note that while the above corollary allows us to conclude the existence of infinitely many root components from the existence of its associated $\delta$-maximal component, the action of the Virasoro does not \textit{a priori} say anything about the appearance of the $\delta$-maximal root components in the tensor product decomposition. Nevertheless, to conclude the proof of Theorem I of Introduction, our work is reduced by Corollary \ref{redux} to showing the existence of these $\delta$-maximal root components. 

Recall that, for any $\beta \in \Phi^+_{Re}$, we can write $\beta=\gamma+k\delta$ for some $\gamma \in \mathring{\Phi}$ and $k \in \Z_{\geq 0}$. This can further be separated into two cases: If $\gamma \in \mathring{\Phi}^+$, then we can take any $k \geq 0$. If $\gamma \in \mathring{\Phi}^-$, then necessarily $k \geq 1$. Therefore, when considering the $\delta$-maximal root components $V(\lambda+\mu-\beta)$, it suffices to consider the two possibilities: $V(\lambda+\mu-\gamma)$ and $V(\lambda+\mu-(-\gamma+\delta))$, where $\gamma \in \mathring{\Phi}^+$. The first of these cases follows immediately from Kumar's result in the finite case \cite{Ku1}.

\begin{Prop} \label{finite} Let $(\lambda, \mu, \gamma)$ be a Wahl triple with $\gamma \in \mathring{\Phi}^+$. Then, $m_{\lambda, \mu}^{\lambda+\mu-\gamma} \neq 0$.
\end{Prop}

\begin{proof}
 Let $\mathring{V}(\lambda) \subset V(\lambda)$ denote the $\mathring{\mf[g]}$-submodule generated by the highest weight vector, and similarly for $\mathring{V}(\mu)$. Then, by \cite{Ku1}, Theorem 1.1, we have 
$$
\mathring{V}(\lambda+\mu-\gamma) \subset \mathring{V}(\lambda) \otimes \mathring{V}(\mu).
$$

Denote by $v_\gamma \in V(\lambda) \otimes V(\mu)$ the highest weight vector for $\mathring{\mf[g]}$ that generates $\mathring{V}(\lambda+\mu-\gamma)$. We claim that this is in fact a highest weight vector for $\mf[g]$. Indeed, all that remains to check is that $e_0.v_\gamma=0$. By the assumption that $\lambda(d)=\mu(d)=0$, we have that no vector in $V(\lambda) \otimes V(\mu)$ can have weight $\nu$ with $\nu(d) \geq 1$. However, since $\gamma \in \mathring{\Phi}^+$, if $e_0.v_\gamma \neq 0$, we have $(\lambda+\mu-\gamma+\alpha_0)(d)=\alpha_0(d) =1$, a contradiction. Thus,  $e_0.v_\gamma=0$ and generates a $\mf[g]$-submodule $V(\lambda+\mu-\gamma) \subset V(\lambda) \otimes V(\mu)$.

\end{proof}

From this, via Corollary \ref{redux},  we have the following corollary. 

\begin{Cor} \label{positivemaximal}
Let $(\lambda, \mu, \beta)$ be a Wahl triple with $\beta=\gamma+k\delta$ for some $\gamma \in \mathring{\Phi}^+$ and $k \geq 0$. Then, $V(\lambda+\mu-\beta) \subset V(\lambda) \otimes V(\mu)$.
\end{Cor}

\section{Existence of  root components  for $\beta = \gamma+k\delta$, for $\gamma \in \mathring{\Phi}^-$}

We now  show the existence of the  $\delta$-maximal root components of the form $V(\lambda+\mu-(-\gamma+\delta))$ for $\gamma \in  \mathring{\Phi}^+$. The results and arguments of this section parallel those in \cite{Ku1}; we reproduce some  of the original arguments therein here for completeness.

Recall from Proposition \ref{mult} that 
$$
m_{\lambda, \mu}^{\lambda+\mu-\beta} = \text{dim}\{v \in V(\mu)_{\mu-\beta}: e_i^{\lambda(\alpha_i^\vee)+1}.v=0 \ \text{for all the simple roots } \alpha_i\}.
$$
Our goal is to explicitly construct such a vector $v \in V(\mu)_{\mu-\beta}$ for $(\lambda, \mu, \beta) \in (\dom)^2 \times \Phi^+_{Re}$ a Wahl triple. To do so, we begin with the following preparatory lemma.

\begin{Lem} \label{root} Let $\beta \in \Phi^+_{Re}$, with $\beta=\gamma+k\delta$ for some $\gamma \in \mathring{\Phi}$ and $k \in \Z_{\geq 0}$, and suppose $\beta-2\alpha_i$ is a root for some $0\leq i \leq \ell$. Then, neither of $\beta-2\alpha_j$ or $\beta-2\alpha_j-\alpha_i$ is a root or zero for $j \neq i$.

\end{Lem}

\begin{proof}
If $\beta \in \mathring{\Phi}^+$ (that is, if $k=0$), then the lemma holds  by explicit knowledge of the finite root systems; see \cite{Bou} for example. 
Else, first suppose $\beta-2\alpha_0$ is a root. Then, we have $\beta-2(\delta-\theta) = (\gamma+2\theta)+(k-2)\delta$ is a root, where $\theta \in \mathring{\Phi}^+$ is the highest root of the underlying semisimple Lie algebra. Thus,  $\beta-2\alpha_0$ is a root if and only if $\gamma+2\theta$ is a root in $\mathring{\Phi}$, so that $\gamma=-\theta$. In this case, neither of $\beta-2\alpha_j$ or $\beta-2\alpha_j-\alpha_0$ is a root for any $j \neq 0$, as the first would mean that $-\theta - 2\alpha_j \in \mathring{\Phi}$, and the second would mean that $-2\alpha_j \in \mathring{\Phi}$, a clear contradiction. Also, clearly, $\beta - 2 \alpha_j$ or  $\beta - 2 \alpha_j -\alpha_i$ can not be zero.

Now, suppose $\beta-2\alpha_i$ is a root for $i \neq 0$. Then, we have $(\gamma-2\alpha_i) + k\delta$ is a root, so that $\gamma - 2\alpha_i \in \mathring{\Phi}$. By the finite case, we get that $\beta-2\alpha_j$ and $\beta -2\alpha_j-\alpha_i$ are not roots or zero for $j \neq 0, i$, as this would have to hold for $\gamma$. (For $\gamma$ positive, we have seen this above and for $\gamma$ negative, we can similarly see using the explicit knowledge of the root system as in \cite{Bou}.)
As above, $\beta - 2\alpha_0 \in \Phi$ would imply that $\gamma=-\theta$, contradicting $\beta-2\alpha_i$ is a root. Finally, $\beta-2\alpha_0-\alpha_i \in \Phi$ would imply that $\gamma -\alpha_i + 2\theta \leq \theta$, so that $\gamma=-\theta$ or $\gamma=-\theta+\alpha_i$; in each of these cases, $\gamma-2\alpha_i$ is not a root, contradicting $\beta-2\alpha_i$ being a root.  Further, it is easy to see that $\beta - 2 \alpha_j -\alpha_i$ can not be zero for any $j\neq i$.
\end{proof}

We next introduce the following set of indices, which will play a role in the construction of sufficient Wahl triples associated to a positive root $\beta$.

\begin{Def} \label{defF} {\rm For $\beta \in \Phi^+$, define $F_\beta:=\{0\leq i \leq \ell: \beta-\alpha_i \not \in \Phi \sqcup \{0\}\}$.}
 \end{Def}

\begin{Prop} \label{oneroot} Let $(\lambda, \mu, \beta)$ be a Wahl triple with $\beta$ a real root, and suppose that $\beta-2\alpha_i \not \in \Phi^+$ for any $ 0\leq i \leq \ell$, or else $\mu(\beta^\vee)=1$. Then, $V(\lambda+\mu-\beta) \subset V(\lambda) \otimes V(\mu)$. 

Thus, for simply-laced $\mf[g]$ (i.e., $\mathring{\mf[g]}$ is simply-laced), for any Wahl triple, 
\begin{equation} \label{eqnwahl} V(\lambda+\mu-\beta) \subset V(\lambda) \otimes V(\mu).
\end{equation}
\end{Prop}

\begin{proof}
Let $v_\mu$ denote a nonzero highest weight vector of $V(\mu)$, and choose a nonzero root vector $X_{-\beta} \in \mf[g]_{-\beta}$. Set $v:= X_{-\beta}.v_\mu$. For any $0 \neq X_\beta \in \mf[g]_\beta$, we have (using \cite{Ku3}, Theorem 1.5.4),
$$
X_\beta.v=X_\beta(X_{-\beta}.v_\mu)= [X_\beta, X_{-\beta}]v_\mu=(X_\beta | X_{-\beta})\nu^{-1}(\beta)\cdot v_\mu=(X_\beta | X_{-\beta}) \mu(\nu^{-1}(\beta))v_\mu , 
$$
where $\nu: \mf[h] \to \mf[h]^*$ is the map induced by $(\cdot | \cdot)$. This is nonzero as $(X_\beta | X_{-\beta}) \neq 0$ and by the regularity condition $(P2)$, $\mu(\nu^{-1}(\beta)) >0$.  Then, $v\in V(\mu)_{\mu-\beta}$ is a nonzero vector. 

By the dominance of $\lambda$ we  have $\lambda(\alpha_i^\vee)+1\geq 1$. Then, if $i \in F_\beta$, we get
$$e_i.v=e_i(X_{-\beta}.v_\mu)=[e_i, X_{-\beta}].v_\mu=0,$$ as by definition $\alpha_i-\beta \not \in \Phi \cup \{0\}$. Else, if $i \not \in F_\beta$, we have by the regularity condition $(P2)$ that $\lambda(\alpha_i^\vee)+1 \geq 2$, and 
$$
e_i^2.v=((\ad_{e_i})^2(X_{-\beta})).v_\mu=0
$$
by the assumption that $\beta-2\alpha_i \not \in \Phi^+$. Thus,  for all $i$, $e_i^{\lambda(\alpha_i^\vee)+1}.v=0$. This proves the proposition in the case $\beta-2\alpha_i \not \in \Phi^+$ for any $  i $ by using Proposition \ref{mult}.

For the proof in the case  $\mu(\beta^\vee)=1$, observe first that for any weight $\nu$ of $V(\mu)$, $(\mu | \mu)\geq (\nu | \nu)$ (cf. \cite{Kac}, Proposition 11.4 (a)). Denoting $n :=\lambda(\alpha_i^\vee) +1$, we have:
\begin{align*} 
(\mu -\beta +n\alpha_i | \mu -\beta +n\alpha_i) &=(\mu   | \mu)+ n^2(\alpha_i | \alpha_i) + 2n (\mu-\beta | \alpha_i),\,\,\text{since $s_\beta \mu = \mu -\beta$}\\
&=(\mu   | \mu)+ n^2(\alpha_i | \alpha_i) + n (\alpha_i | \alpha_i) (\mu-\beta | \alpha_i^\vee)\\
&=(\mu   | \mu)+ n (\alpha_i | \alpha_i) \left((\lambda+ \mu-\beta | \alpha_i^\vee)+1\right)\\
&> (\mu   | \mu),\,\,\,\text{since $\lambda+\mu-\beta\in \dom$ by (P1)}.
\end{align*}
Thus, $V(\mu)_{\mu -\beta +n\alpha_i} = 0.$

To prove the equation \eqref{eqnwahl}, by Corollaries \ref{ImRootComp}, 
\ref{redux} and \ref{positivemaximal}, it suffices to assume that 
$\beta = \delta - \gamma$, for a positive root $\gamma \in  \mathring{\Phi}^+$.  Any such root $\beta$ does satisfy the condition that $\beta - 2 \alpha_i \notin \Phi^+$ for any simple root $\alpha_i$. Hence the first part of the proposition proves \ref{eqnwahl}. 
\end{proof}

\begin{Rmk} {\rm This proof holds more generally for \textit{any} symmetrizable Kac-Moody algebra $\mf[g]$ and triple $(\lambda, \mu, \beta)$ satisfying $(P1)$ and $(P2)$ under  the  condition on $\beta$ as in Proposition \ref{oneroot}.} \end{Rmk}

Next, we consider triples $(\lambda, \mu, \beta)$ where $\beta$ is a real root with $\beta-2\alpha_i \in \Phi^+$ for some $0\leq i \leq \ell$. By Lemma \ref{root}, this $i$ is unique. For technical reasons, we assume that $\mf[g] \neq G_2^{(1)}$; this will be handled separately. 

\begin{Prop} \label{tworoot} Let $\mf[g] \neq G_2^{(1)}$, and let $(\lambda, \mu, \beta)$ be a Wahl triple with $\beta$ a real root. Suppose $\alpha_i$ is the unique simple root such that $\beta-2\alpha_i \in \Phi^+$. In particular, $\mu(\beta^\vee) \geq 1$ by the condition (P2).
Then, $V(\lambda+\mu-\beta) \subset V(\lambda) \otimes V(\mu)$ assuming at least one of the following: 
\begin{enumerate}
\item $\lambda \in \mathcal{P}^{+}$ is \textbf{regular} dominant, or 
\item $F_{\beta}=F_{\beta-\alpha_i}$
\end{enumerate}
\end{Prop}

\begin{proof}
Again, let $v_\mu \in V(\mu)$ denote a nonzero highest weight vector, $X_{-\beta} \in \mf[g]_{-\beta}$ a nonzero root vector, and define $X_{-\beta+\alpha_i}:=[e_i, X_{-\beta}]$, $X_{-\beta+2\alpha_i}:=[e_i, X_{-\beta+\alpha_i}]$. Since $\mf[g] \neq G_2^{(1)}$, we have that $|\beta(\alpha_i^\vee)|\leq 2$, so that $\beta-3\alpha_i, \beta+\alpha_i \not \in \Phi$. Set
$$
v:=(X_{-\beta+\alpha_i}f_i-2\mu(\alpha_i^\vee)X_{-\beta})\cdot v_\mu \in V(\mu)_{\mu-\beta}.
$$
Then, by a similar computation as in the proof of Proposition \ref{oneroot} (making use of the condition (P2) since $i\notin F_\beta$), we have  $X_\beta.v \neq0$ for any $0 \neq X_{\beta} \in \mf[g]_\beta$, so that $v \neq 0$. (Since $\beta+\alpha_i \notin \Phi, \beta(\alpha_i^\vee)\geq 0$.)
 Now, for any $j \neq i$, we have 
$$
\begin{aligned}
e_j^2.v&=e_j^2(X_{-\beta+\alpha_i}f_i.v_\mu)-2\mu(\alpha_i^\vee)e_j^2(X_{-\beta}.v_\mu) \\
&=((\ad_{e_j})^2(X_{-\beta+\alpha_i})f_i).v_\mu-2\mu(\alpha_i^\vee)((\ad_{e_j})^2(X_{-\beta})).v_\mu \\
&=0
\end{aligned}
$$
as $(\ad_{e_j})^2(X_{-\beta+\alpha_i})=0$ and $(\ad_{e_j})^2(X_{-\beta})=0$ by Lemma \ref{root}. Further, we have 
$$
\begin{aligned}
e_i.v&=X_{-\beta+2\alpha_i}f_iv_\mu+\mu(\alpha_i^\vee)X_{-\beta+\alpha_i}v_\mu-2\mu(\alpha_i^\vee)X_{-\beta+\alpha_i}v_\mu \\
&=X_{-\beta+2\alpha_i}f_iv_\mu-\mu(\alpha_i^\vee)X_{-\beta+\alpha_i}v_\mu \\
\implies e_i^2.v &= \mu(\alpha_i^\vee)X_{-\beta+2\alpha_i}v_\mu-\mu(\alpha_i^\vee)X_{-\beta+2\alpha_i}v_\mu \\
&=0
\end{aligned}
$$

Therefore $e_j^2.v=0$ for all $j$. \\

If $\lambda \in \mathcal{P}^{+}$ is regular dominant, then $\lambda(\alpha_j^\vee)+1 \geq 2$ for all $j$. So, by Proposition \ref{mult} we get $V(\lambda+\mu-\beta) \subset V(\lambda) \otimes V(\mu)$. 

Else, for any $j$ such that $\lambda(\alpha_j^\vee)=0$, by condition (P2) we have $j \in F_\beta$; note that $i \not \in F_\beta$. Thus,  if $F_\beta=F_{\beta-\alpha_i}$, we have 
$$
e_j.v=[e_j, X_{-\beta+\alpha_i}]f_i\cdot v_\mu-2\mu(\alpha_i^\vee)[e_j, X_{-\beta}]\cdot v_\mu = 0,
$$
so, again by Proposition \ref{mult}, we have $V(\lambda+\mu-\beta)\subset V(\lambda) \otimes V(\mu)$.

\end{proof}

To end this section, recall the set of indices $F_\beta$ for $\beta \in \Phi^+$ from Definition \ref{defF}. The following dominant weights will play a crucial role for the remaining constructions. 

\begin{Def} \label{rhobeta} 
{\rm Given $\beta \in \Phi^+_{Re}$, define $\rho_\beta:= \sum_{i \not \in F_\beta} \Lambda_i$. }
\end{Def}

\noindent By construction, the weight $\rho_\beta$ satisfies condition (P2) for the root $\beta$ and is the \textit{minimal} dominant weight with $\rho_\beta(d)=0$ that can do so. Therefore, we can conclude the following lemma by the additivity property of the tensor decomposition Corollary \ref{additive}.

\begin{Lem} \label{Newlemma}
If $2\rho_\beta-\beta \in \dom$ is such that 
$
m_{\rho_\beta, \rho_\beta}^{2\rho_\beta-\beta} \not = 0$, then $m_{\lambda, \mu}^{\lambda+\mu-\beta} \neq 0
$
for any Wahl triple $(\lambda, \mu, \beta)$. 

Further, it is easy to see that $2\rho_\beta -\beta \in \mathcal{P}^{+}$ for any $\mf[g]$ except possibly for $G_2$-type $\mathring{\mf[g]}$.
\end{Lem}

\section{Exceptional Maximal Components}\label{exceptional}

In light of Corollaries \ref{ImRootComp},
\ref{redux} and  \ref{positivemaximal}, all that remains to show to complete the proof of Theorem I is the existence of root components $V(\lambda+\mu-\beta) \subset V(\lambda) \otimes V(\mu)$ for $\beta=-\gamma+\delta$, where $\gamma \in \mathring{\Phi}^+$. We can reduce further, by making use of Proposition \ref{oneroot} and Lemma \ref{root}, to the case that $\beta-2\alpha_i \in \Phi^+$ for some unique $i$  (in particular, $\mathring{\mf[g]}$ is not simply-laced). We can exclude the case $i=0$, via the following easy lemma. 

\begin{Lem}\label{lem9.1}
Let $(\lambda, \mu, \beta)$ be a Wahl triple such that $\beta=-\gamma+k\delta$ for some $\gamma \in \mathring{\Phi}^+$ and $k \geq 1$ and such that $\beta-2\alpha_0 \in \Phi$. Then, $V(\lambda+\mu-\beta) \subset V(\lambda) \otimes V(\mu)$. 
\end{Lem}

\begin{proof}
Direct computation shows that 
$$
\beta-2\alpha_0 = -\gamma+k\delta-2\alpha_0=2\theta-\gamma+(k-2)\delta.
$$
This is a root if and only if $2\theta-\gamma \in \mathring{\Phi}$, so necessarily $\gamma=\theta$. But, then we have $\beta=\alpha_0+(k-1)\delta$. Thus,  we can reduce to the $\delta$-maximal component associated to $\beta=\alpha_0$; the existence of $V(\lambda+\mu-\alpha_0)$ in $V(\lambda) \otimes V(\mu)$ is clear by taking any  linear combination: $c(f_0\cdot v_\lambda)\otimes v_\mu + d v_\lambda\otimes f_0\cdot v_\mu$, for $c,d\in \mathbb{C}$.
\end{proof}

The remaining options for such $\beta = -\gamma + \delta$ (where $\gamma \in \mathring{\Phi}^+$), with $i \neq 0$, correspond to roots $\gamma \in \mathring{\Phi}^+$ such that $\gamma+2\alpha_i \in \mathring{\Phi}^+$. We collect these possibilities, organized by the type $\mf[g]$, $\gamma \in \mathring{\Phi}^+$, and the associated unique simple root denoted $\alpha_i(\gamma)$, in Table \ref{Exceptional} below.

\begin{table}[h!]
\centering
\begin{tabular}{||c|c|c||} 
\hline
$\mf[g]$ & $\gamma$ & $\alpha_j(\gamma)$ \\ 
\hline
$B_\ell^{(1)}$ & $\sum_{i \leq k < \ell} \alpha_k, \ 1 \leq i < \ell$ & $\alpha_\ell$  \\
$C_\ell^{(1)}$ & $\left( 2 \sum_{i \leq k < \ell} \alpha_k \right) + \alpha_\ell, \ 2 \leq i \leq \ell$ & $\alpha_{i-1}$ \\
$F_4^{(1)}$ & $\alpha_2$ & $\alpha_3$ \\
$F_4^{(1)}$ & $\alpha_1+\alpha_2$ & $\alpha_3$ \\
$F_4^{(1)}$ & $\alpha_2 + 2\alpha_3$ & $\alpha_4$ \\
$F_4^{(1)}$ & $\alpha_1+\alpha_2+2\alpha_3$ & $\alpha_4$ \\
$F_4^{(1)}$ & $\alpha_1+2\alpha_2+2\alpha_3$ & $\alpha_4$ \\
$F_4^{(1)}$ & $\alpha_1+2\alpha_2+2\alpha_3+2\alpha_4$ & $\alpha_3$ \\
$G_2^{(1)}$ & $\alpha_2$ & $\alpha_1$ \\
$G_2^{(1)}$ & $\alpha_1+\alpha_2$ & $\alpha_1$ \\
\hline

\end{tabular}
\caption{Exceptional Components}
\label{Exceptional}
\end{table}

We refer to these roots $\beta$ and their associated root components as "exceptional," since root behavior of these  do not appear in the semisimple setting and as the methods to demonstrate their existence in the tensor decomposition do not uniformly fit into the general approach of \cite{Ku1}. Instead, we will primarily make use of the PRV  components to construct these final $\delta$-maximal components. {\it Throughout the remainder of this section, we deal with  Wahl triples $(\lambda, \mu, \beta =\delta-\gamma)$ for each specified $\gamma$ in the above table.} We first recall the following results:

\begin{Thm} \label{PRV} (\cite{Ku2}, Theorem 3.7; \cite{M2}, Corollaire 3)(PRV components) Let $\mf[g]$ be any symmetrizable Kac-Moody Lie algebra. For any $\lambda, \mu \in \dom$ and the weyl group elements $v, w \in W$  such that $\eta:= v\lambda+w\mu \in \dom$, we have $V(\eta) \subset V(\lambda) \otimes V(\mu)$.
\end{Thm}

\subsection{\bf $B_\ell^{(1)}$}
For $\gamma= \sum_{i \leq k < \ell} \alpha_k, \ 1 \leq i < \ell$, $\beta=-\gamma+\delta$, we have $\alpha_\ell$ as the unique simple root such that $\beta-2\alpha_\ell \in \Phi^+$. Setting $[\ell]:=\{0,1,2,\dots,\ell\}$, we have:
$$
F_\beta= \begin{cases}
[\ell]\backslash \{i-1, \ell\} & i \neq 2 \\
[\ell] \backslash \{0,1,\ell\} & i =2 
\end{cases}
$$
Comparing this to $F_{\beta-\alpha_\ell}$, we see that $F_\beta=F_{\beta-\alpha_\ell}$ in all the cases; thus we can apply Proposition \ref{tworoot}. 

\subsection{\bf $C_\ell^{(1)}$} For any $\gamma = \gamma_i  = (2\sum_{i\leq k<\ell}\alpha_k) + \alpha_\ell \,(2\leq i \leq \ell)$ in the above table, $F_{\beta_i} =[\ell]\setminus \{i-1\}$, where $\beta_i := \delta - \gamma_i$. 
 In this case $\rho_{\beta_i}(\beta_i^\vee) = 1$. Thus, Proposition \ref{oneroot} and Lemma \ref{Newlemma} take care of this case.

\subsection{\bf $F_4^{(1)}$}
\begin{itemize}

\item $\gamma=\alpha_2$: Here, $\beta=-\alpha_2+\delta=\alpha_0+2\alpha_1+2\alpha_2+4\alpha_3+2\alpha_4$, $F_\beta=\{0,2,4\}$, and we can take $\lambda=\mu=\rho_\beta=\Lambda_1+\Lambda_3$. We use the following variant of the  construction from \cite{Ku1}: define $v \in V(\rho_\beta)_{\rho_\beta-\beta}$ by
$$
v:= \left( X_{-\beta+\alpha_3}f_3 -2X_{-\beta}-X_{-\beta+\alpha_3+\alpha_4} [f_3, f_4] \right) v_{\Lambda_1+\Lambda_3}.
$$
Then, one can check directly as previously that $X_\beta.v\neq 0$, so that $v \neq 0$, and that $e_0.v=e_2.v=e_4.v=e_1^2.v=e_3^2.v=0$, as required to apply Proposition \ref{mult}. \\
\vspace{1em}

\item $\gamma=\alpha_1+\alpha_2$: We have $\beta=-\gamma+\delta=\alpha_0+\alpha_1+2\alpha_2+4\alpha_3+2\alpha_4$, $F_\beta = \{1, 2, 4\}$, and we can take $\lambda=\mu=\rho_\beta=\Lambda_0+\Lambda_3$. Set $w=s_3s_2s_4s_3s_2s_4s_3$, $v=s_1s_0$. Then, 
$$
v\rho_\beta=\rho_\beta -\alpha_0-\alpha_1, \ w\rho_\beta=\rho_\beta-2\alpha_2-4\alpha_3-2\alpha_4.
$$
Thus, we have $V(w\rho_\beta+v\rho_\beta)=V(2\rho_\beta-\beta) \subset V(\rho_\beta) \otimes V(\rho_\beta)$ as a  PRV component. 
\vspace{1em}

\item $\gamma=\alpha_2+2\alpha_3$: In this case, we have $\beta=-\gamma+\delta=\alpha_0+2\alpha_1+2\alpha_2+2\alpha_3+2\alpha_4$, and $\alpha_4$ is the unique simple root such that $\beta-2\alpha_4 \in \Phi^+$. Then, $F_\beta=\{0,2,3\}=F_{\beta-\alpha_4}$ and  we can apply Proposition \ref{tworoot}. 
\vspace{1em}

\item $\gamma=\alpha_1+\alpha_2+2\alpha_3$: We have $\beta=-\gamma+\delta=\alpha_0+\alpha_1+2\alpha_2+2\alpha_3+2\alpha_4$, and $\alpha_4$ is the unique simple root such that $\beta-2\alpha_4 \in \Phi^+$. As in the previous case, we compute $F_\beta=\{1,3\} = F_{\beta-\alpha_4}$, so that we can apply Proposition \ref{tworoot}. 
\vspace{1em}

\item $\gamma=\alpha_1+2\alpha_2+2\alpha_3$: For this case,  we get $\beta=-\gamma+\delta=\alpha_0+\alpha_2+\alpha_2+2\alpha_3+2\alpha_4$. Set $w=s_4s_3s_1s_2s_3s_4$, $v=s_0$, and use $\lambda=\mu=\rho_\beta=\Lambda_0+\Lambda_4$. Then, 
$$
v\rho_\beta = \Lambda_0+\Lambda_4-\alpha_0,\,\,w\rho_\beta= \rho_\beta-\alpha_1-\alpha_2-2\alpha_3-2\alpha_4,
$$
so that we get $V(w\rho_\beta+v\rho_\beta)=V(2\rho_\beta-\beta) \subset V(\rho_\beta) \otimes V(\rho_\beta)$ as a PRV component. 
\vspace{1em}

\item $\gamma=\alpha_1+2\alpha_2+2\alpha_3+2\alpha_4$: In this final case, we have $\beta=-\gamma+\delta=\alpha_0+\alpha_1+\alpha_2+2\alpha_3$, and $\alpha_3$ is the unique simple root such that $\beta-2\alpha_3 \in \Phi^+$. Then, we see that $F_\beta=\{1,2,4\}=F_{\beta-\alpha_3}$, so we can apply Proposition \ref{tworoot}. 

\end{itemize}

\subsection{\bf $G_2^{(1)}$}
\begin{itemize}

\item $\gamma=\alpha_2$: We have $\beta=-\alpha_2+\delta=\alpha_0+3\alpha_1+\alpha_2$. Then, the minimal dominant $\lambda, \mu \in \mathcal{P}^+$ satisfying the necessary conditions (P1) and (P2) are $\lambda=\Lambda_0+\Lambda_1$, $\mu=\Lambda_0+2\Lambda_1$. Setting $w=s_1s_2s_1$, $v=s_0$, we have 
$$
v\mu = \Lambda_0+ 2 \Lambda_1-\alpha_0,\,\, w\lambda=s_1s_2s_1(\Lambda_0+\Lambda_1) = \lambda-3\alpha_1-\alpha_2,
$$
so that we have $V(w\lambda+v\mu)=V(\lambda+\mu-\beta) \subset V(\lambda) \otimes V(\mu)$ as a PRV component. \vspace{1em}

\item $\gamma=\alpha_1+\alpha_2$: We have $\beta=-\alpha_1-\alpha_2+\delta=\alpha_0+2\alpha_1+\alpha_2$, $F_\beta=\{2\}$. Then, we have $\rho_\beta=\Lambda_0+\Lambda_1$, and $2\rho_\beta-\beta \in \mathcal{P}^+$. Setting $w=s_2s_1$ and $v=s_0s_1$, we have $w\rho_\beta= \Lambda_0+\Lambda_1-\alpha_1-\alpha_2$ and $v\rho_\beta=\Lambda_0+\Lambda_1-\alpha_0-\alpha_1$, so that we have $V(w\rho_\beta+v\rho_\beta) = V(2\rho_\beta-\beta) \subset V(\rho_\beta) \otimes V(\rho_\beta)$ as a PRV component. 
\end{itemize}

\section{Main Theorem on Root Components}

Combining Corollaries \ref{ImRootComp}, \ref{redux} and \ref{positivemaximal}, Lemma \ref{Newlemma} and Proposition \ref{oneroot}
 and the results from Section \ref{exceptional}, we get the following main result of our paper (cited as Theorem I in  Introduction):

Recall that by a Wahl triple we mean a triple $(\lambda, \mu, \beta) \in ({\dom})^2 \times \Phi^+$ such that 
\begin{itemize}
\item[(P1)] $\lambda+\mu-\beta \in \dom$, and 
\item[(P2)] If $\lambda(\alpha_i^\vee)=0$ or $\mu(\alpha_i^\vee)=0$, then $\beta- \alpha_i \not \in \Phi \sqcup \{0\}$.
\end{itemize} 

\begin{Thm} \label{mainthm}For any affine Lie algebra $\mf[g]$ and Wahl triple $(\lambda, \mu, \beta) \in ({\dom})^2 \times \Phi^+$, $$V(\lambda+\mu-\beta) \subset V(\lambda) \otimes V(\mu).$$
\end{Thm}
\newpage

\section{Study of the cohomology $H^*(X_w^\cP, \mathscr{I}_e^2\otimes \mathscr{L}_w(\mu))$}

Let $\g$ be a symmetrizable Kac-Moody Lie algebra with its standard Borel subalgebra $\fb$ and standard Cartan subalgebra $\h$ (\cite{Ku3}, $\S$1.1.2).  Let $\Phi\subset \h^*$ be the set of roots  and $\Phi^+$ (resp. $\Phi^-$) be the subset of positive (resp. negative) roots and let $\Delta =\{\alpha_1, \dots, \alpha_\ell\}\subset \Phi^+$ be the 
subset of simple roots with $\alpha_i^\vee$ the corresponding simple coroots. Let  $\dom$ be the set of dominant integral weights. 
The integrable, highest weight (irreducible) $\mf[g]$-modules are parameterized by $\dom$. For $\lambda \in \dom$, let $V(\lambda)$ be the corresponding integrable, highest weight $\mf[g]$-module. For any suset $S\subset \{1, \dots, \ell\}$, let $\p=\p_S :=\fb \oplus (\oplus_{\alpha\in \Phi^+_S}\g_{-\alpha})$ be the corresponding standard parabolic subalgebra, where $\g_\alpha$ is the root space corresponding to the root $\alpha$ and $\Phi^+_S :=
\Phi^+\cap (\oplus_{k\in S}\,\mathbb{Z}_{\geq 0}\alpha_k)$. Then, 
$$\fu := \oplus_{\alpha\in \Phi^+(S)}\, \g_\alpha$$
is the nil-radical of $\p$, where $\Phi^+(S):= \Phi^+ \setminus \Phi^+_S.$
Let $\G, \B, \cP=\cP_S, \cU, \cH$ be the corresponding `maximal' groups associated to the Lie algebras $\g, \fb, \p, \fu, \h$ respectively (cf. \cite{Ku3}, Section 6.1). Let $W$ be the Weyl group of $\g$ and let $W_\cP'$ be the set of smallest length coset representatives in $W/W_\cP$, where $W_\cP$ is the subgroup of $W$ generated by the simple reflections $\{s_k\}_{k\in S}$. 

 For any pro-algebraic $\cP$-module $M$, by $\cL(M)$ we 
mean the corresponding homogeneous vector bundle (locally free sheaf) on $\X_\cP:=\G/\cP$ associated to the 
principal $\cP$-bundle: $\G \to \G/\cP$ by the representation $M$ of $\cP$ (cf. \cite{Ku3}, Corollary 8.2.5).
For any integral weight $\lambda \in \h^*$, such that $\lambda(\alpha_k^\vee)= 0$ for all $k \in S$, the one 
dimensional $\cH$-module $\mathbb{C}_\lambda$  (given by the character $\lambda$) admits a unique $\cP$-module 
structure (extending the $\cH$-module structure); in particular, we have the 
line bundle $\cL(\C_\lambda)$ on $\X_\cP$. We abbreviate the line bundle $\cL(\C_{-\lambda})$
 by $\cL(\lambda)$ and its restriction to the Schubert variety $X_w^\cP :=\overline{\B w\cP/\cP}$ by $\cL_w(\lambda)$ (for any $w\in W_\cP'$). Given two line bundles  $\cL(\lambda)$ and   $\cL(\mu)$, we can form their 
external tensor product to get the line bundle $\cL(\lambda\boxtimes \mu)$ on $\X_\cP\times \X_\cP$.

A dominant integral weight $\mu$ is called {\it $S$-regular} if $\mu(\alpha_k^\vee) = 0$ if and only if $k\in S$. The set of such weights is denoted by ${\dom_S}^o$.
    With this definition, we have the following:
\begin{Prop}  \label{prop11.1} For any $\mu\in {\dom_S}^o$ and $w\in W_\cP'$ such that $X_w^\cP$ is $\cP$-stable unde the left multiplication (i.e., $w^{-1}\alpha_k\in \Phi_S^+ \sqcup \Phi^-$ for all the simple roots $\alpha_k$ with $k\in S$),
\vskip1ex
    (a) $H^p(X_w^\cP, \cI_e \otimes \cL_w(\mu)) = 0$, for all $p > 0$,
\vskip1ex

    (b) $H^p(X_w^\cP, \cI_e^2 \otimes \cL_w(\mu)) = 0$, for all $p > 0$  , and
\vskip1ex
    (c) $H^0(X_w^\cP, \cI_e^2 \otimes \cL_w(\mu))$ (resp. $H^0(X_w^\cP, (\mathscr{O}_w/\cI_e^2) \otimes \cL_w(\mu))$ 
     is canonically isomorphic (as $\cP$-modules) with $\left(V_w(\mu)/(\hat{T}_e(X_w^\cP) \cdot v_\mu)\right)^*$ (resp.
      $(\hat{T}_e(X_w^\cP)      
       \cdot v_\mu)^*)$, 
     where $\cI_e$ denotes the ideal sheaf of the base point $e = \cP/\cP \subset  X_w^\cP, \mathscr{O}_w$  denotes the 
structure sheaf of $X_w^\cP$, $v_\mu\neq 0\in [V(\mu)]_\mu$ , $V_w(\mu) \subset V(\mu)$ is the Demazure module  (which is the $\fb$-submodule of $V(\mu)$ generated by the extremal weight vector $v_{w\mu}$) and 
$\hat{T}_e(X_w^\cP)      
       \cdot v_\mu := \mathbb{C} v_\mu\oplus T_e(X_w^\cP)\cdot v_\mu$ (here $T_e(X_w^\cP)$ is  the Zariski tangent space of $X_w^\cP$ at $e$).

    (Observe that, since $\mu(\alpha_k^\vee) = 0$ for $k\in S$, $\hat{T}_e(X_w^\cP) \cdot v_\mu  \subset V(\mu)$
     is stable under $\cP$, and also the point $e$ being $\cP$-fixed, $H^0(X_w^\cP, \cI_e^k \otimes \cL_w(\mu))$ and $H^0(X_w^\cP, (\mathscr{O}_w/\cI_e^k) \otimes \cL_w(\mu))$   have canonical $\cP$-module structures for any $k\geq 1$.)
\end{Prop}\begin{proof} (a) Consider the sheaf exact sequence:
\begin{equation} \label{eqn11.1}
0\to \cI_e\otimes \cL_w(\mu) \to \cL_w(\mu)\to  (\cO_w/\cI_e)\otimes \cL_w(\mu)\to 0.
\end{equation}
Clearly, $H^p(X_w^\cP, (\mathscr{O}_w/\cI_e) \otimes \cL_w(\mu))= H^p(e, \cL(\mu)_{|e})=0$, for all $p > 0$. 
Also, $H^p(X_w^\cP,  \cL_w(\mu))=0$, for $p>0$, and the restriction map:   $H^0(X_w^\cP, \cL_w(\mu))\to 
H^0(e, \cL_w(\mu)_{|e})$ is surjective (cf. \cite{Ku3}, Theorem 8.2.2). So, the long exact cohomology sequence  associated to the 
sheaf sequence \eqref{eqn11.1} gives (a).

\vskip1ex

     (b) Similarly, consider the sheaf exact sequence:
    \begin{equation} \label{eqn11.2}
0\to \cI_e^2\otimes \cL_w(\mu) \to \cL_w(\mu)\to  (\cO_w/\cI_e^2)\otimes \cL_w(\mu)\to 0.
\end{equation} 
     We write a part of the corresponding cohomology exact sequence:
 \begin{align} \label{eqn11.3}
0\to H^0(X_w^\cP, \cI_e^2\otimes \cL_w(\mu)) \to &H^0(X_w^\cP,  \cL_w(\mu))\overset{\pi}\to  H^0(X_w^\cP, (\cO_w/\cI_e^2)\otimes \cL_w(\mu))\notag\\ 
&\to 
H^1(X_w^\cP, \cI_e^2\otimes \cL_w(\mu))\to 0.
\end{align} 
The sheaf exact sequence:
\begin{equation} \label{eqn11.4}
0\to (\cI_e/\cI_e^2)\otimes \cL_w(\mu)\to (\cO_w/\cI^2_e) \otimes \cL_w(\mu) \to (\cO_w/\cI_e)\otimes \cL_w(\mu)  \to 0
\end{equation} 
gives:
\begin{equation} \label{eqn11.5}
0\to T_e(X_w^\cP)^*\otimes \C_{\mu}^*
\to H^0(X_w^\cP, (\cO_w/\cI^2_e) \otimes \cL_w(\mu)) \to  \C_{\mu}^*
 \to 0.
\end{equation} 

Now, we assert that the map $\pi$ is surjective: To prove 
this, observe first  that the map 
$$\theta: \fu^-\to V(\mu), \,\,\,X \mapsto X\cdot v_\mu ,$$
is injective, where $\fu^-$ is the opposite nil-radical of $\p$: For any  root $\beta \in  \Phi^+(S)$, pick  root vectors $X_\beta\in \g_\beta$ and $Y_\beta \in \g_{-\beta}$. Then, 
$$X_\beta Y_\beta\cdot v_\mu= (X_\beta | Y_\beta) (\mu |\beta)v_\mu,\,\,\,\text{by \cite{Ku3}, Theorem 1.5.4}.$$
From the non-degeneracy of $(\cdot | \cdot)$ over $\g$ and the weight consideration, we get the injectivity of $\theta$. (Here we have used the assumption that $\mu \in  {\dom_S}^o$.) 
Since $\pi$ is a $\p$ and hence $\h$-module map, from the exact sequence \eqref{eqn11.5} and the following identification with the dual Demazure module:
$$H^0(X_w^\cP \cL_w(\mu))\simeq V_w(\mu)^*\,\,\,\text{(cf. \cite{Ku3}, Corollary 8.1.26)},
$$
to prove the surjectivity of $\pi$, it suffices to show that $\theta_w: T_e(X_w^\cP) \to V_w(\mu), \,X\mapsto X\cdot v_\mu ,$
is injective, where we identify $T_e(X_w^\cP) \hookrightarrow T_e(X^\cP)\simeq \fu^-$. But, the injectivity of $\theta_w$ clearly follows from the injectivity of $\theta$.

        So, the exact sequence \eqref{eqn11.3} (and the surjectivity of $\pi$) establishes 
the vanishing of $H^1(X_w^\cP, \cI_e^2\otimes \cL_w(\mu))$. Vanishing of $H^p(X_w^\cP, \cI_e^2\otimes \cL_w(\mu))$,
for $p>1$,  
 follows by considering the cohomology exact sequence associated to the sheaf exact sequence:
\begin{equation} \label{eqn11.6}
0\to \cI_e^2\otimes \cL_w(\mu) \to \cI_e\otimes \cL_w(\mu)\to  (\cI_e/\cI_e^2)\otimes \cL_w(\mu)\to 0,
\end{equation} 
and the (a)-part of the proposition. So, (b) follows.

\vskip1ex
        (c) It is easy to see, from the description of the map $\pi$, that
\begin{equation} \label{eqn11.7}
\Ker \pi =\{f\in V_w(\mu)^*: f_{|(\hat{T}_e(X_w^\cP)\cdot v_\mu)}=0\}\simeq \left(V_w(\mu)/(\hat{T}_e(X_w^\cP)\cdot v_\mu)\right)^*.\end{equation} 
So, using the exact sequences \eqref{eqn11.3} and  \eqref{eqn11.5}, (c) follows.                                                               
\end{proof}

From the above proposition together with \cite{Ku3}, Theorem 8.2.2, we immediately get:

     \begin{Cor}   \label{remark11.1} Under the assumptions of Proposition \ref{prop11.1}, $H^p(X_w^\cP, (\cO_w/\cI^k_e)\otimes \cL_w(\mu))= 0$, for all $p > 0$
      and $k = 1, 2$.
      \end{Cor}
      
    \section{Surjectivity of the Gaussian map}  
      
     We continue to follow the notation and assumptions from Section 11.    For any $w\in W_\cP'$ such that  the Schubert variety $X_w^\cP$ is $\cP$-stable, set
      $$\hat{\X}_w^\cP:= \G\times^{\cP} X_w^\cP.$$ 
      Consider the isomorphism:
      $$\delta:\G\times^\cP \X_\cP \simeq \X_\cP \times \X_\cP,\,\,\,[g, x]\mapsto (g\cP, gx),\,\,\text{for $g\in \G$ and $x\in \X_\cP$}.$$
    We have an embedding   
       $$\hat{\X}_w^\cP \hookrightarrow    \G\times^\cP \X_\cP .$$  
       In particular, we can restrict the line bundle $\cL(\lambda\boxtimes \mu)$ via the above isomorphism $\delta$ to 
      $\hat{\X}_w^\cP$ to get the line bundle denoted   $\cL_w(\lambda\boxtimes \mu)$.      
   \begin{Lem} \label{lem12.1} For $\lambda, \mu \in  {\dom_S}^o, k=1,2, w\in W_\cP'$ as above,   
   and any $p\geq 0$, there are 
canonical isomorphisms:
\begin{equation*}
H^p(\hat{\X}_w^\cP, \hat{\cI}^k_e\otimes \cL_w(\lambda\boxtimes \mu))\simeq
H^p(\X_\cP, \cL(\lambda)\otimes \cL(H^0(X_w^\cP, \cI^k_e\otimes \cL_w(\mu)))),
\end{equation*}
and
\begin{equation*}
H^p(\hat{\X}_w^\cP, (\hat{\cO}_w/\hat{\cI}^k_e)\otimes \cL_w(\lambda\boxtimes \mu))\simeq
H^p(\X_\cP, \cL(\lambda)\otimes \cL(H^0(X_w^\cP, (\cO_w/\cI^k_e) \otimes \cL_w(\mu)))),
\end{equation*}
where $\hat{\cI}_e$ denotes the ideal sheaf of $\widehat{\X}_e^\cP$ in  $ \hat{\X}_w^\cP$
and  $\hat{\cO}_w$ is the structure sheaf of $\hat{\X}_w^\cP$.

Further, 
\begin{equation} \label{eqn12.1}H^p(\hat{\X}_w^\cP, \hat{\cI}_e\otimes \cL_w(\lambda\boxtimes \mu))=0,\,\,\,\text{for all $p>0$}.
\end{equation}

        (We will see later that $H^p(\hat{\X}_w^\cP, \hat{\cI}_e^2\otimes \cL_w(\lambda\boxtimes \mu))=0,$ for all $p>0$.)       
    \end{Lem}
    \begin{proof}
        Consider the fibration 
        $$\sigma: \hat{\X}_w^\cP\to   \X_\cP, \,\,\,[g, x]\mapsto      g\cP$$ 
        and the associated Leray spectral sequence. Then, it is easy to see that $R^p\sigma_*(\hat{\cI}_e^k\otimes \cL_w(\lambda\boxtimes \mu))$ is canonically isomorphic with   $ \cL(\lambda)\otimes \cL(H^p(X_w^\cP, \cI^k_e\otimes \cL_w(\mu)))$ for any $p\geq 0$. 
 (A similar statement is true for  $\hat{\cI}_e^k$ replaced by $\hat{\cO}_w/\hat{\cI}_e^k $.)
 In particular, by 
Proposition \ref{prop11.1}, Corollary  \ref{remark11.1} and the (degenerate) Leray spectral sequence, first 
part of the lemma follows.

     To prove the vanishing \eqref{eqn12.1}, use the cohomology exact sequence associated to the sheaf exact sequence:
$$0\to \hat{\cI}_e\otimes \cL_w(\lambda\boxtimes \mu)\to \cL_w(\lambda\boxtimes \mu)\to  (\hat{\cO}_w/\hat{\cI}_e)\otimes \cL_w(\lambda\boxtimes \mu) \to 0,$$ 
and \cite{Ku2}, Theorem 2.7 (also see \cite{M2}, Lemmas 21, 22).
\end{proof}
                                             
     The following lemma is used in proving our crucial Proposition \ref{prop12.1}.

     \begin{Lem} \label{lem12.2} Let $\g$ be an affine Kac-Moody Lie algebra (cf. Section 3). Take  $\lambda, \mu \in  {\dom_S}^o $ and $\theta \in \dom$.    
      Then, if $\Hom_{\p}\left(V(\theta)^\vee, \C_{\lambda}^*\otimes (\g\cdot v_\mu)^\vee\right)$  is 
      nonzero, then  $\theta = \lambda +\mu-\beta$, for some $\beta \in \Phi^+(S)\cup \{0\}$.
     \vskip1ex
      
   If $\beta$ is zero, then the above space is one dimensional.    
   \vskip1ex
       
 If such a $\beta$ is a real root, then the above space is nonzero if and only if $S\subset F_\beta$ (cf Definition \ref{defF}), in which case it is one dimensional.
 \vskip1ex
 
If $\beta$ is an imaginary root, then this space is of dimension $(\dim \g_\beta) - |S|$.
\end{Lem}
\begin{proof} It is easy to see that, as a $\p$-module,
$$ \g\cdot v_\mu \simeq (\g/\Ann v_\mu)\otimes \C_\mu,$$
where $\Ann v_\mu$ is the annihilator of $v_\mu$ in $\g$ and the right side is given the 
tensor product $\p$-module structure.  Hence, using the nondegenerate invariant  form $(. | .)$ on $\g$, we 
get:
\begin{equation}\label{eqn12.2}
(\g\cdot v_\mu)^\vee \simeq (\fu \oplus (\Ker \mu)^\perp) \otimes \C_{\mu}^*\,,
\end{equation}
where $\Ker \mu:= \{h\in \h : \mu(h)) = 0\}$, 
$(\Ker \mu)^\perp :=\{h\in \h: (h| \Ker \mu)=0\}$ and $\fu \oplus (\Ker \mu)^\perp$ is 
a $\p$-module under the adjoint representation. By the presentation of $V(\theta)^\vee$ as a $U(\fb)$-module
as in the proof of Proposition \ref{mult}, 
 we get:
\begin{align*}
&\Hom_\fb\left(V(\theta)^\vee, \C_{\lambda}^*\otimes (\g\cdot v_\mu)^\vee\right)
\simeq
 \Hom_\fb\left(V(\theta)^\vee, \C_{(\lambda +\mu)}^*\otimes (\fu \oplus (\Ker \mu)^\perp)\right),\,\,\text{by}\, \eqref{eqn12.2}\\
&\simeq \{v\in [\fu \oplus (\Ker \mu)^\perp]_{\lambda+\mu-\theta}: (\ad e_i)^{\theta(\alpha_i^\vee)+1}v = 0\}, 
\,\, \text{for all the simple roots\, $\alpha_i$},
 \end{align*}             
  where $e_i$ is a root vector corresponding to the simple root $\alpha_i$. It is easy to see, from the above description, that
\begin{align}\label{eqn12.3}
\Hom&_\p\left(V(\theta)^\vee, \C_{\lambda}^*\otimes (\g\cdot v_\mu)^\vee\right)\simeq 
\{v\in [\fu \oplus (\Ker \mu)^\perp]_{\lambda
+\mu-\theta}: (\ad e_i)^{\theta(\alpha_i^\vee)+1}v = 0,  \\ \notag
& \,\,\text{for all the simple roots $\alpha_i$ and $(\ad f_k)v=0$ for all $k\in S$}\},
 \end{align}  
 where $f_k$ is a root vector corresponding to the negative simple root     $-\alpha_k$.     
Clearly,  the following 
conditions (1) and (2) are both satisfied if the right side of \eqref{eqn12.3} is nonzero:
\vskip1ex

    (1) $\lambda +\mu -\theta =\beta$, for some $\beta \in \Phi^+(S)\cup \{0\}$, and  
    (2) If $\beta \neq 0$ is real, $F_\beta \supset S$.
   \vskip1ex
    
    Conversely, for any dominant $\theta$ which satisfies the conditions (1) 
and (2) as above, the right side of \eqref{eqn12.3} (and hence 
$\Hom_\p\left(V(\theta)^\vee, \C_{\lambda}^*\otimes (\g\cdot v_\mu)^\vee\right)$
 is one dimensional if either $\beta=0$ or $\beta$ is a real root:
 
 \vskip1ex
 
    If $\beta =0$, $(\ad e_i)^{\theta(\alpha_i^\vee)+1}((\Ker \mu)^\perp)=0 $ for all $i \notin S$,   
      since $\lambda, \mu \in  {\dom_S}^o $. Moreover,  $\ad e_k  ((\Ker \mu)^\perp) = \ad f_k   ((\Ker \mu)^\perp)=0$    
    for all $k \in S$ (since $\mu(\alpha_k^\vee)=0$).
      
    So,  assume now that $\beta\in \Phi^+(S)$ is a real root and pick $0\neq X_\beta\in \g_\beta$.  By the above condition (2), $(\ad f_k) X_\beta = 0$, for $k\in S$. So, by the $sl(2)$-module theory, $\beta(\alpha_k^\vee)\leq 0$ and 
    $(\ad e_k)^{1-\beta(\alpha_k^\vee)}X_\beta =0$. In particular,  $(\ad e_k)^{1+\theta(\alpha_k^\vee)}X_\beta =0.$   For $i\notin S$, we again claim that $(\ad e_i)^{\theta(\alpha_i^\vee)+1}X_\beta =0$, since $\beta + (\theta(\alpha_i^\vee)+1)\alpha_i\notin \Phi^+$: We have 
    $$    (\beta + (\theta(\alpha_i^\vee)+1)\alpha_i)  (\alpha_i^\vee)= (\lambda +\mu +\theta )(\alpha_i^\vee)+ 2\geq 4 \,,$$
    since  $\lambda, \mu \in  {\dom_S}^o $ and $\theta \in \dom$.
    But, this contradicts \cite{Bou},  Page 278, Fact 6. (This is where we use the assumption on $\g$ being an affine Kac-Moody Lie algebra.)
    This completes the proof of the lemma in the case $\beta$ is either zero or a real root.                          

Assume now that $\beta = \lambda +\mu -\theta$ is an imaginary root $p\delta, p>0$. For any $h\in \h$ take $v(h)=h\otimes t^p\in \g_{p\delta}$. Then, for any $k\in S$,
$$(\ad e_k) v(h) = (\ad f_k) v(h) =0 \,\,\text{if and only if}\,\, \alpha_k(h)= 0,$$
where $\alpha_0$ is interpreted as the negative of the highest root. Moreover, for $i\notin S$,
\begin{equation} \label{eqn11.15} (\ad e_i)^{(\lambda +\mu-p\delta)(\alpha_i^\vee) +1} v(h)=0,\,\,\,\text{since $\lambda, \mu \in  {\dom_S}^o$ and 
$p\delta(\alpha_i^\vee) = 0$}.
\end{equation}
Thus, by the equation \eqref{eqn12.3}, for $\theta = \lambda +\mu-p\delta$,
\begin{equation}\label{eqn12.4}
\Hom_\p\left(V(\theta)^\vee, \C_{\lambda}^*\otimes (\g\cdot v_\mu)^\vee\right)\simeq 
\{h\otimes t^p: \alpha_k(h) = 0 \ \forall  k\in S\}.
\end{equation}
In particular, it has dimension $\dim \g_{p\delta} - |S|$.
\end{proof}
\begin{Lem} \label{lem12.3} Let the notation and assumptions be as in Lemma \ref{lem12.2} and let $M$ be a $\p$-module quotient 
$\pi: R \twoheadrightarrow M$, where $R:=  \C_{\lambda}^*\otimes (\g\cdot v_\mu)^\vee$. Then, for any $\theta \in \dom$, the canonical map 
\begin{equation} \label{eqn12.5}
\hat{\pi}_M: \Hom_\p\left(V(\theta)^\vee, R\right) \to  \Hom_\p\left(V(\theta)^\vee, M\right)
\end{equation}
is surjective. 
\end{Lem}
\begin{proof} By the proof of Lemma \ref{lem12.2}, both the sides of \eqref{eqn12.5} are zero unless $\theta =\lambda +\mu -\beta$, for some $\beta \in \Phi^+(S)\cup \{0\}$. We identify $  \Hom_\p\left(V(\theta)^\vee, R\right) $ with the right side of \eqref{eqn12.3} and a similar identification for  $  \Hom_\p\left(V(\theta)^\vee, M\right)$. To prove the lemma, we consider the three cases separately:

\vskip1ex

\noindent
{\bf Case I, $\beta =0$:} In this case  $  \Hom_\p\left(V(\theta)^\vee, R\right)\simeq (\Ker \mu)^\perp$ (cf. Proof of Lemma \ref{lem12.2}) and thus the lemma follows easily in this case.

\vskip1ex

\noindent
{\bf Case II,  $\beta$ is a real root:} In this case, following the proof of Lemma \ref{lem12.2},
\begin{align*}   \Hom_\p\left(V(\theta)^\vee, R\right)&\simeq \g_\beta,\,\,\,\text{if $F_\beta \supset S$}\\
&= 0, \,\,\,\text{otherwise}.
\end{align*}
Identifying $R$ with $\fu\oplus (\Ker \mu)^\perp$ via the equation \eqref{eqn12.2}, it suffices to prove that if $0\neq \pi(\g_\beta) \subset   \Hom_\p\left(V(\theta)^\vee, M\right)$, then $F_\beta \supset S$. Since $0\neq \pi(\g_\beta) \subset   \Hom_\p\left(V(\theta)^\vee, M\right)$, we get that, for all $k\in S$, $[f_k, \pi(\g_\beta)] = 0$, which implies $\beta (\alpha_k^\vee)\leq 0$. Assune, if possible,  that $F_\beta \not\supset S$, i.e., for some $k\in S, \beta -\alpha_k\in \Phi^+$ (it can not be zero since $\beta \in \Phi^+(S)$). Then, $[e_k, \g_{\beta - \alpha_k}]$ is nonzero and hence, $\g_\beta$ being one dimensional,  $[e_k, \g_{\beta - \alpha_k}] = \g_\beta$ giving rise to 
 $[e_k, \pi(\g_{\beta - \alpha_k})] \neq 0$. From this, by the $sl_2$-module theory, we easily see that $[f_k, \pi(\g_{\beta})] \neq 0$. A contradiction (to the equation \eqref{eqn12.3}) since $\pi (\g_\beta) \subset 
\Hom_\p\left(V(\theta)^\vee, M\right)$. Thus, $F_\beta \supset S$. From this the lemma follows in this case. 
\vskip1ex

\noindent
{\bf Case III,  $\beta$ is an imaginary root $p\delta$ for $p>0$:} Take $h\otimes t^p$ (for $h\in \h$) such that 
$\pi (h\otimes t^p) \subset \Hom_\p\left(V(\theta)^\vee, M\right)$. Define the subset $S'\subset S$ by
$$S':= \{k\in S: \pi (\alpha_k^\vee \otimes t^p)=0\},$$
where we set $\alpha_0^\vee = -\theta^\vee$ for the highest root $\theta$  of $\mathring{\g}$. Consider the decomposition 
\begin{equation} \label{eqn12.6} \mathring{\h}= \left(\oplus_{k\in S'}\,\C \alpha_k^\vee\right) \oplus \mathring{\h}(S'),
\end{equation}
where $\mathring{\h}(S') := \{x\in \mathring{\h}: \alpha_k(x)=0 \ \forall k\in S'\}.$
Let $h'$ be the projection of $h$ in the $\mathring{\h}(S')$-factor (under the above decomposition). Then, from the definition of $S'$, we see that $\pi (h\otimes t^p) = \pi (h'\otimes t^p).$ Moreover, it is easy to see that $
h'\otimes t^p \in   \Hom_\p\left(V(\theta)^\vee, R\right)$, i.e., 
$$(\ad e_k)(h'\otimes t^p) = (\ad f_k)(h'\otimes t^p) =0,\,\,\,\text{for all $k\in S$}$$
 (cf. the equation  \eqref{eqn11.15}).
 \end{proof}
 We prove the 
following crucial result towards the proof of  Theorem \ref{thm12.1}. We follow the notation as in 
Lemma \ref{lem12.1}.

\begin{Prop} \label{prop12.1} Let $\g$ be an affine Kac-Moody Lie algebra (cf. Section 3). Then, for any   $\lambda, \mu \in  {\dom_S}^o $ (where $S$ is an arbitrary subset of the simple roots of $\g$) and  any $w\in W_\cP'$ such that  the Schubert variety $X_w^\cP$ is $\cP$-stable, consider the following two conditions:
  \vskip1ex
  
    (a) $H^1(\hat{\X}_w^\cP, \hat{\cI}_e^2\otimes \cL_w(\lambda\boxtimes \mu)) = 0.$
    \vskip1ex
    
    (b) For all the real roots $\beta \in \Phi^+(S)$, satisfying $F_\beta \supset S$ and $\lambda +\mu -\beta \in \dom$
    (in particular, $(\lambda, \mu, \beta)$ is a Wahl triple), there exists a 
        $f_\beta \in \Hom_\fb\left(\C_{\lambda +\mu -\beta}\otimes V(\lambda)^\vee, V(\mu)\right)$ such that 
        \begin{equation} \label{eqn12.4new} X_\beta(f_\beta(\C_{\lambda +\mu -\beta}\otimes  v_\lambda^*))\neq 0,\,\,\text{for $X_\beta\neq 0 \in \g_\beta$},\,\,\,\text{
        where $v_\lambda^* \neq 0\in [V(\lambda)^\vee]_{-\lambda}$}.
        \end{equation}
        
        Then, the condition (b) implies the condition (a). 
        \end{Prop}
  \begin{proof}      
  Consider the sheaf exact sequence over $\hat{\X}_w^\cP$:
\begin{equation}\label{eqn12.14}
0\to \hat{\cI}^2_e\otimes \cL_w(\lambda\boxtimes \mu)\to \cL_w(\lambda\boxtimes \mu)\to  (\hat{\cO}_w/\hat{\cI}_e^2)\otimes \cL_w(\lambda\boxtimes \mu) \to 0.
\end{equation}

From the cohomology exact sequence associated to the above sheaf exact sequence and the vanishing of 
 $H^1( \hat{\X}^\cP_w,  \cL_w(\lambda\boxtimes \mu))$ (cf. \cite{Ku2}, Theorem 2.7), we get that (a) is equivalent to the surjectivity of the canonical map 
 $$\tau_w: H^0( \hat{\X}_w^\cP,  \cL_w(\lambda\boxtimes \mu))\to H^0( \hat{\X}_w^\cP,  (\hat{\cO}_w/\hat{\cI}_e^2) \otimes \cL_w(\lambda\boxtimes \mu)).$$
 
   By Lemma \ref{lem12.1}  and Proposition \ref{prop11.1} we get: 
\begin{equation*} 
H^0(\hat{\X}_w^\cP, (\hat{\cO}_w/\hat{\cI}^2_e) \otimes \cL_w(\lambda\boxtimes \mu))
\simeq H^0(\X_\cP, \cL(\lambda)\otimes \cL((\hat{T}_e(X_w^\cP)\cdot v_\mu)^*)).
\end{equation*}
Further, by \cite{Ku3}, Theorem 8.2.2,
\begin{equation*}
H^0(\hat{\X}_w^\cP,  \cL_w(\lambda\boxtimes \mu))\simeq H^0( \X_\cP,  \cL(\lambda)\otimes \cL(V_w(\mu)^*)).
\end{equation*}
      Transporting the map $\tau_w$ under these identifications, we get the map
          $$\tilde{\tau}_w: H^0( \X_\cP,  \cL(\lambda)\otimes \cL(V_w(\mu)^*))\to H^0(\X_\cP, \cL(\lambda)\otimes \cL((\hat{T}_e(X_w^\cP)\cdot v_\mu)^*)),
          $$
          which is induced from the canonical restriction: $ \cL(V_w(\mu)^*)\to \cL((\hat{T}_e(X_w^\cP)\cdot v_\mu)^*)$.       
          
   By \cite{M1}, Proposition 12, for any finite dimensional pro-algebraic $\cP$-module $M$, we have a 
canonical $\G$-module isomorphism:
$$H^0(\X_\cP, \cL(M))\simeq \oplus_{\theta \in \dom}\,V(\theta)^*\otimes [V(\theta)\otimes M]^\p,$$
where $\G$ acts trivially on $ [V(\theta)\otimes M]^\p$. 
So, the surjectivity of the map $\tilde{\tau}_w$  (and hence of $\tau_w$) is equivalent to the surjectivity 
of the canonical restriction maps
$$\gamma_\theta^w: \Hom_\p\left(V(\theta)^\vee, \C_{\lambda}^*\otimes V_w(\mu)^*\right)\to 
 \Hom_\p\left(V(\theta)^\vee, \C_{\lambda}^*\otimes (\hat{T}_e(X_w^\cP)\cdot v_\mu)^* 
 \right)$$
 for all $\theta \in \dom$.

For $\theta =\lambda+\mu$, the map $\gamma^w_\theta$ is clearly nonzero and hence surjective by Lemmas 
\ref{lem12.2} and \ref{lem12.3}.

 Making use of Lemmas \ref{lem12.2} and \ref{lem12.3} again, we can assume
 (for the surjectivity of $\gamma_\theta^w$)
 that $\theta$ is of the form $\theta = \lambda+\mu -\beta$ for  $\beta \in \Phi^+(S)$ and if $\beta$ is a real root, then $F_\beta\supset S$. 
 
 We first analyze the surjectivity of 
 $$\gamma_\theta: \Hom_\p\left(V(\theta)^\vee, \C_{\lambda}^*\otimes V(\mu)^\vee\right)\to 
 \Hom_\p\left(V(\theta)^\vee, \C_{\lambda}^*\otimes (\g\cdot v_\mu)^\vee 
 \right).$$ 
 Choose a positive definite Hermitian form $\{ \cdot | \cdot \}$ on $V(\lambda)$  (and 
$V(\mu)$) satisfying 
$$\{Xv| w\} = \{v| \sigma(X)w\}, \,\,\text{for $v, w \in V(\lambda)$  and $X \in \g$},$$
where $\sigma$  is a conjugate-linear involution of $\g$ which takes $\g_\alpha$  to $\g_{-\alpha}$ for 
any root $\alpha$  (cf. \cite{Ku3}, Theorem 2.3.13). Now, set the tensor product form (again denoted by) $\{ \cdot | \cdot
\}$ on $V(\lambda) \otimes V(\mu)$.

   Consider the following diagram (for $\theta =\lambda +\mu-\beta$):
   \[
\xymatrix{
\Hom_\g(V(\theta), V(\lambda)\otimes V(\mu))\ar[d]_\simeq^{\gamma_1}\ar[r]_\simeq^{\gamma_2} & \Hom_\g(V(\lambda)\otimes V(\mu), V(\theta))
\ar[r]_\simeq^{\gamma_3} & \Hom_\p(\C_{\lambda}\otimes V(\mu), V(\theta))
\ar[d]^{\hat{\gamma}_\theta}\\
\Hom_\fb(\C_{\theta}\otimes V(\lambda)^\vee, V(\mu))
 &  &  \Hom_\p(\C_{\lambda}\otimes (\g\cdot v_\mu), V(\theta))\,,
}
\]
where $\gamma_1$ and $ \gamma_3$ are the canonical restriction maps, which are isomorphisms (as can be easily seen).  We now describe the map $\gamma_2$: For any nonzero $f\in \Hom_\g(V(\theta), V(\lambda)\otimes V(\mu))$, write $V(\lambda) \otimes V(\mu) = \Image f \oplus (\Image f)^\perp$ and set $\gamma_2(f)$ as the projection on the first factor (identifying it with $V(\theta)$ under $f$). It is easy to see that $\gamma_2$ is a bijection (though non-linear). Further, $\hat{\gamma}_\theta$ is the canonical restriction map. 

We identify 
$$  \Hom_\p\left(\C_{\lambda}\otimes V(\mu), V(\theta)\right)
\simeq  \Hom_\p\left(V(\theta)^\vee, \C_\lambda^*\otimes V(\mu)^\vee\right)$$
and 
$$  \Hom_\p\left(\C_{\lambda}\otimes (\g\cdot v_\mu), V(\theta)\right)
\simeq  \Hom_\p\left(V(\theta)^\vee, \C_\lambda^*\otimes (\g\cdot v_\mu)
^\vee\right).$$

 We first take $\theta  = \lambda +\mu -\beta$ for a real root $\beta \in \Phi^+(S)$ such that $S \subset F_\beta$. 
Following through the various isomorphisms in the above diagram, it can be seen that (denoting $\fn^+ =[\fb, \fb]$ and $\fn^-$ is the opposite subalgebra to $\fn$)
\begin{align*}
\hat{\gamma}_\theta & \text{ is  nonzero (and hence surjective since the range has dimension $1$)} \\
& \Leftrightarrow \,\text{there exists a $\fb$-morphism $f=f_\beta: \C_\theta\otimes V(\lambda)^\vee \to V(\mu)$}\\ 
& \,\,\,\,\,\, \,\,\,\,\,\text{such that $\{(\gamma_1^{-1}f) V(\theta)| \C_\lambda\otimes (\g\cdot v_\mu)\}\neq 0$}\\
& \Leftrightarrow \{U(\mathfrak{n}^-)((\gamma_1^{-1}f) \C_\theta) | \C_\lambda\otimes (\g\cdot v_\mu)\}\neq 0\\
& \Leftrightarrow \{(\gamma_1^{-1}f) \C_\theta | \C_\lambda\otimes (\g\cdot v_\mu)\}\neq 0\,\,\text{by the invariance of}\\
&\,\,\,\,\,\, \,\,\,\,\,\,\{\cdot | \cdot\}; \,\text{since $\C_\lambda\otimes (\g \cdot v_\mu)$ is $U(\mathfrak{n})$-stable}\\
& \Leftrightarrow \{f (\C_\theta\otimes v_\lambda^*) | \g\cdot v_\mu\}\neq 0 \\
& \Leftrightarrow \{f (\C_\theta\otimes v_\lambda^*) | X_{-\beta}\cdot v_\mu\}\neq 0, \,\text{for $0\neq X_{-\beta}\in \g_{-\beta}$} \\
& \Leftrightarrow \{X_\beta(f (\C_\theta\otimes v_\lambda^*)) | v_\mu\}\neq 0, \,\text{by the invariance of $\{\cdot | \cdot\}$} \\
& \Leftrightarrow X_\beta(f (\C_\theta\otimes v_\lambda^*))\neq 0.
\end{align*}
For $\beta = 0$, $\hat{\gamma}_\theta$ is clearly surjective. 

So, let us now take $\beta$ to be an imaginary root $p\delta$, for $p>0$:

Recall from the isomorphism \eqref{eqn12.4} that, under the identification \eqref{eqn12.2}, 
$$ \Hom_\p\left(V(\theta)^\vee, \C_\lambda^*\otimes (\g\cdot v_\mu)^\vee\right) \simeq \{h\otimes t^p: h\in \h\,\text{and}\, \alpha_k(h)=0 \ \forall k\in S\}.$$
We next show that, under the restriction map 
$$V(\mu)^\vee \to (\g \cdot v_\mu)^\vee\simeq (\fu \oplus (\Ker \mu)^\perp)\otimes \C_\mu^*,
$$ for any $h\in \h$, the element $x_h :=  (-\frac{h\otimes t^p}{p\mu(K)})\cdot v_\mu^*$
\begin{equation} \label{eqn12.7} x_h \mapsto (h\otimes t^p)\otimes v_\mu^*\,.
\end{equation}
To see this, take any $h'\in \h$. Then, 
\begin{align*} x_h(h'\otimes t^{-p})\cdot v_\mu)&= v_\mu^*\left(\left[\frac{h\otimes t^p}{p\mu(K)}, h'\otimes t^{-p}\right]\cdot v_\mu\right)\\
&=(h|h').
\end{align*}
Similarly, 
\begin{align*} \left((h\otimes t^p)\otimes v_\mu^*\right)((h'\otimes t^{-p})\cdot v_\mu)&= (h\otimes t^p| h'\otimes t^{-p})
\\
&=(h|h').
\end{align*}
This proves \eqref{eqn12.7}. Now, for any $h\in \h$ such that $\alpha_k(h)=0$ for all $k\in S$, the element 
$x_h$ satisfies 
\begin{equation}\label{eqn12.8} e_i^{\theta(\alpha_i^\vee)+1}\cdot x_h=0\, \forall i,\,\,\,\text{and}\,\, f_k\cdot x_h =0 \,\forall k\in S.
\end{equation}
To prove this, observe that (in $V(\mu)^\vee$)
\begin{align*} 
e_i^{\theta(\alpha_i^\vee)+1}\cdot (h\otimes t^p)\cdot v_\mu^* &=\sum_{j=0}^{\theta(\alpha_i^\vee)+1} \,\begin{pmatrix}\theta(\alpha_i^\vee)+1\\j \end{pmatrix} \left((\ad e_i)^j (h\otimes t^p)\right)\cdot (e_i^{\theta(\alpha_i^\vee)+1-j}\cdot v_\mu^*),\\
& \,\,\,\,\,\,\,\,\,\,\,\, \text{cf. the proof of \cite{Ku3}, Lemma 1.3.3}\\
&=0,\,\,\,\text{since $(\ad e_i)^2 (h\otimes t^p)=0$ and $e_i^{\mu(\alpha_i^\vee) +1}\cdot v_\mu^* =0$},
\end{align*}
by the presentation of $V(\mu)^\vee$ as a $\fb$-module given in the proof of Proposition \ref{mult}. Further, for $k\in S$,
\begin{align*} 
f_k\cdot (h\otimes t^p)\cdot v_\mu^* &= [f_k, h\otimes t^p]\cdot v_\mu^*
\\
&=0,\,\,\,\text{since $\alpha_k(h)=0$ by assumption.}
\end{align*}
Thus, $v_\theta^* \mapsto x_h$ extends uniquely to a $\p$-module morphism $ V(\theta)^\vee \to \C_\lambda^*\otimes V(\mu)^\vee$. This proves the surjectivity of the restriction map 
$$\gamma_\theta: \Hom_\p\left(V(\theta)^\vee, \C_\lambda^*\otimes V(\mu)^\vee\right) \to \Hom_\p\left(V(\theta)^\vee, \C_\lambda^*\otimes (\g\cdot v_\mu)^\vee\right)$$
for $\theta = \lambda+\mu-p\delta$.

Consider the commutative diagram for any $\theta\in \dom$ of the form $\theta = \lambda+ \mu -\beta$, for $\beta \in \Phi^+(S)\cup \{0\}$ and $w\in W_\cP'$ such that $X_w^\cP$ is $\cP$-stable:

 \[
\xymatrix{
\Hom_\p\left(V(\theta)^\vee, \C_\lambda^*\otimes V(\mu)^\vee\right)\ar[d] \ar[r]^{\gamma_\theta} & \Hom_\p\left(V(\theta)^\vee, \C_\lambda^*\otimes (\g\cdot v_\mu)^\vee\right)
\ar[d]^{\hat{\pi}_M}\\
\Hom_\p\left(V(\theta)^\vee, \C_\lambda^*\otimes V_w(\mu)^*\right)
  \ar[r]_{\gamma_\theta^w} 
  &  \Hom_\p\left(V(\theta)^\vee, \C_{\lambda}^*\otimes (\hat{T}_e(X_w^\cP)\cdot v_\mu)^*\right),
}
\]
where the vertical maps are the canonical restriction maps and $M := \C_{\lambda}^*\otimes (\hat{T}_e(X_w^\cP)\cdot v_\mu)^*$. By Lemma \ref{lem12.3}, $\hat{\pi}_M$ is surjective. Further, by the above proof, $\gamma_\theta$ 
is surjective for all $\theta = \lambda+ \mu -\beta$ such that $\beta =0$ or an imaginary root. Moreover, as proved above, for a real root $\beta \in \Phi^+(S)$ with $F_\beta \supset S$,
$\gamma_\theta$ is surjective if  and only if there exists a $\fb$-morphism $f_\beta: \C_\theta\otimes V(\lambda)^\vee \to V(\mu)$ such that $X_\beta(f_\beta(\C_\theta \otimes v_\lambda^*))\neq 0$.

By Lemma \ref{lem12.2}, for any real root $\beta \in \Phi^+(S)$ satisfying  $F_\beta \not\supset S$, we have
$\Hom_\p\left(V(\theta)^\vee, \C_\lambda^*\otimes (\g\cdot v_\mu)^\vee\right) =0$. In particular, $\gamma_\theta^w$ is surjective for $\beta =0$ or an imaginary root or a real root  $\beta \in \Phi^+(S)$ such that $F_\beta \not\supset S$ and it is surjective for a real root $\beta \in \Phi^+(S)$ with $F_\beta \supset S$ if the condition (b) of the proposition is satisfied for $\beta$. 
Thus, $\tilde{\tau}_w$
(equivalently $\tau_w$)  is surjective  if the condition (b) of the proposition is satisfied. This proves the proposition.
 \end{proof}

\begin{Lem} \label{lem12.5} Let $(\lambda, \mu, \beta)$ be a Wahl triple for a real root $\beta$ and let $V(\lambda+\mu-\beta) \subset V(\lambda) \otimes V(\mu)$ be a $\delta$-maximal root component. Observe that $\beta \in \mathring{\Phi}^+$ or $\beta=\delta-\gamma$ for $\gamma \in \mathring{\Phi}^+$. Then, 
\begin{enumerate}
\item If $\beta \in \mathring{\Phi}^+$, the validity of condition (b) of Proposition \ref{prop12.1} for $V(\lambda+\mu-\beta)$
(i.e., the validity of \eqref{eqn12.4new}) 
 implies its validity for $V(\lambda+\mu-\beta-k\delta)$ for any $k \geq 0$. 

\item If $\beta=\delta-\gamma$ for $\gamma \in \mathring{\Phi}^+$, then we have (with notation as introduced below in the proof , see the identity \eqref{eqn12.6.10}) 
$$
\begin{aligned}
(l+m+h^\vee) X_{-\gamma}(k+1)\cdot &v_\lambda^*(L_{-k}w)
=l(k+1)X_{-\gamma}(1)\cdot w_{\mu-\beta} \\
-&(\lambda_{|\mathring{\h}}, \gamma_{|\mathring{\h}})X_{-\gamma}(1)\cdot w_{\mu-\beta}+\sum_{\substack{\beta_i \in \mathring{\Phi}^+ :\\ \beta_i+\gamma \in \mathring{\Phi}^+}}[X_{-\gamma}(1), X_{-\beta_i}]X_{\beta_i}\cdot w_{\mu-\beta}.
\end{aligned}
$$
\end{enumerate}
\end{Lem} 
\begin{proof}

Let $w \in V(\lambda+\mu-\beta) \subset V(\lambda) \otimes V(\mu)$ be the highest weight vector. Then, $L_{-k}(w)$ is a highest weight vector in $V(\lambda) \otimes V(\mu)$ with weight $\lambda+\mu-\beta-k\delta$ (cf. \S 4 together with Proposition 7.2).  Let $l$, $m$ be the central charges of $V(\lambda)$ and $V(\mu)$, respectively, and $h^\vee$ the dual Coxeter number of $\mf[g]$. Write 
$$
w=v_\lambda \otimes w_{\mu-\beta} + \sum_{q \in Q^+ \backslash \{0\}} v^\clubsuit_{\lambda-q} \otimes w^\clubsuit_{\mu-\beta+q},
$$
where $v^\clubsuit_{\lambda-q}$ is a basis of weight vectors of $V(\lambda)$ of weight $\lambda-q$, $w^\clubsuit_{\mu-\beta+q} \in V(\mu)_{\mu-\beta+q}$ (possibly zero), and $Q^+:= \bigoplus_{j=0}^\ell \Z_{\geq 0 } \alpha_j$. Observe that $v^\clubsuit_{\lambda-q}$ occurs as many times (possibly zero) as the dimension of the weight space $V(\lambda)_{\lambda-q}$. 

Denote by $v^\ast_\lambda: V(\lambda) \otimes V(\mu) \to V(\mu)$ the contraction via the highest weight vector $v_\lambda$. Then, by the proof of Proposition \ref{mult}, we have that $v^\ast_\lambda(w)$, for $w$ the highest weight vector as above, satisfies the conditions $e_i^{\lambda(\alpha_i^\vee)+1} \cdot (v^\ast_\lambda(w)) =0$ for all simple roots $\alpha_i$. 

By \cite{KRR}, Proposition 10.3, for any $k >0$, the GKO operator (cf. Proposition 4.5): 
\begin{multline*}
(l+m+h^\vee)L_{-k}(w)=m(L_{-k}^{\mf[g]} \otimes Id)(w) + l v_\lambda \otimes (L_{-k}^{\mf[g]} w_{\mu-\beta}) + l \sum_{q \neq 0} v^\clubsuit_{\lambda-q} \otimes (L_{-k}^{\mf[g]} w^\clubsuit_{\mu-\beta+q}) \\
 -\sum_{j \in \Z, i} (u_i(-j) \cdot v_\lambda) \otimes (u^i(j-k) \cdot w_{\mu-\beta}) - \sum_{j \in \Z, \,i} \sum_{ 0 \neq q \leq \beta} (u_i(-j) \cdot v^\clubsuit_{\lambda-q}) \otimes (u^i(j-k) \cdot w^\clubsuit_{\mu-\beta+q}),
\end{multline*}
where $\{u_i\}$ is a basis of $\mathring{\mf[g]}$, $\{u^i\}$ is its dual basis, and $L_{-k}^{\mf[g]}$ is the Virasoro operator of $\mf[g]$ (cf \cite{KRR}, Corollary 10.1). Now, take a root basis $\{X_{\beta_i}\}$ of $\mathring{\mf[n]}^+ :=\oplus_{\alpha\in \mathring{\Phi}^+}\,\mathring{\g}_\alpha$ and the dual basis $\{X_{-\beta_i}\}$ of $\mathring{\mf[n]}^- :=\oplus_{\alpha\in \mathring{\Phi}^+}\,\mathring{\g}_{-\alpha}$, and a basis $\{h_j\}$ of $\mathring{\mf[h]}$ and its dual basis $\{h^j\}$, with respect to the normalized form. That is, $(X_{\beta_i}, X_{-\beta_j})=\delta_{ij}$, $(h_i, h^j)=\delta_{ij}$. Then, we have 
\begin{align} \label{eqn12.5.1}
(l+m+h^\vee) v_\lambda^\ast (L_{-k} (w)) = l \ L_{-k}^{\mf[g]} (w_{\mu-\beta}) &- \sum_{i} v_\lambda^\ast (u_i \cdot v_\lambda)(u^i(-k) \cdot w_{\mu-\beta}) \notag\\
&- \sum_{i, 0 \neq q \leq \beta} v_\lambda^\ast (u_i \cdot v^\clubsuit_{\lambda-q})(u^i(-k) \cdot w^\clubsuit_{\mu-\beta+q}).
\end{align}
Now, we consider the two cases separately. First, assume $\beta \in \mathring{\Phi}^+$. Then, for $X_\beta(k):= X_\beta \otimes t^k$, we have by the equation \eqref{eqn12.5.1}:
\begin{align}\label{eqn12.5.new}
&(l+m+h^\vee)X_\beta(k) \cdot v_\lambda^\ast(L_{-k}(w)) = lkX_\beta \cdot w_{\mu-\beta} +\sum_{j} \lambda(h_j) \beta(h^j) X_\beta \cdot w_{\mu-\beta} \notag\\
&-\sum_{\substack{\beta_i \in \mathring{\Phi}^+ :\\   \beta-\beta_i \in \mathring{\Phi}^+}} v_\lambda^\ast ( X_{\beta_i} \cdot v^\clubsuit_{\lambda-\beta_i}) [X_\beta, X_{-\beta_i}] \cdot w^\clubsuit_{\mu-\beta+\beta_i} 
-v_\lambda^\ast(X_\beta \cdot v^\clubsuit_{\lambda-\beta})((\mu,\beta)+km)w^\clubsuit_\mu, \notag\\
&\,\,\,\,\,\,\,\,\,\,\,\,\,\,\,\text{by \cite{KRR}, Corollary 10.1, Identity 10.13(b)}  \notag\\
&= lkX_\beta \cdot w_{\mu-\beta}+(\lambda_{|\mathring{\mf[h]}}, \beta_{|\mathring{\mf[h]}}) X_\beta \cdot w_{\mu-\beta} - \sum_{\substack{\beta_i \in \mathring{\Phi}^+: \\   \beta-\beta_i \in \mathring{\Phi}^+}} v^\ast_\lambda( X_{\beta_i} \cdot v^\clubsuit_{\lambda-{\beta_i}}) [X_\beta, X_{-\beta_i}] \cdot w^\clubsuit_{\mu-\beta+\beta_i} \notag\\
&-v^\ast_{\lambda}(X_\beta \cdot v^\clubsuit_{\lambda-\beta})((\mu,\beta)+km)w^\clubsuit_\mu.
\end{align}

For any root $\beta_i \in \mathring{\Phi}^+$, choose a basis of $V(\lambda)_{\lambda-\beta_i}$ consisting of $X_{-\beta_i}.v_\lambda$ (which we will denote simply by $v^\circ_{\lambda-\beta_i}$) and any basis $v^{\clubsuit '}_{\lambda-\beta_i}$ of $V(\lambda)_{\lambda-\beta_i}$ annihilated by $X_{\beta_i}$. With such a basis, the above equation \eqref{eqn12.5.new} becomes:

\begin{multline} \label{eqn12.5.2}
(l+m+h^\vee)X_{\beta}(k) \cdot v^\ast_\lambda(L_{-k}(w)) = lkX_\beta \cdot w_{\mu-\beta}+(\lambda_{|\mathring{\mf[h]}}, \beta_{|\mathring{\mf[h]}})X_\beta \cdot w_{\mu-\beta} \\ 
-\sum_{\substack{\beta_i \in \mathring{\Phi}^+:\\   \beta-\beta_i \in \mathring{\Phi}^+}} v^\ast_\lambda(X_{\beta_i}X_{-\beta_i} \cdot v_\lambda)[X_\beta, X_{-\beta_i}] \cdot w^\circ_{\mu-\beta+\beta_i} 
-v^\ast_\lambda(X_\beta X_{-\beta} \cdot v_\lambda)((\mu,\beta)+km)w^\circ_{\mu}.
\end{multline}
Since $X_{\beta_i}$ annihilates $w$, we get 
\begin{equation} \label{eqn12.5.3}
-X_{\beta_i} \cdot w_{\mu-\beta} = (\lambda, \beta_i) w^\circ_{\mu-\beta+\beta_i}.
\end{equation}

Combining the equations \eqref{eqn12.5.2} and \eqref{eqn12.5.3}, we get taking $X_\beta \cdot w_{\mu-\beta}=w_\mu$ (which is possible by the assumption that condition (b) of Proposition \ref{prop12.1} is valid for the $\delta$-maximal component $V(\lambda+\mu-\beta)$):
\begin{multline} \label{eqn12.5.7}
(l+m+h^\vee)X_{\beta}(k) \cdot v^\ast_\lambda(L_{-k}(w)) = lkw_\mu + (\lambda_{|\mathring{\mf[h]}}, \beta_{|\mathring{\mf[h]}})w_\mu \\
+\sum_{\substack{\beta_i \in \mathring{\Phi}^+: \\   \beta-\beta_i \in \mathring{\Phi}^+}} [X_\beta, X_{-\beta_i}]X_{\beta_i} \cdot w_{\mu-\beta} + ((\mu,\beta)+km)w_\mu.
\end{multline}
We next claim that for any $\beta_i \in \mathring{\Phi}^+$ such that $\beta-\beta_i \in \mathring{\Phi}^+$, 
\begin{equation} \label{eqn12.5.6}
[X_\beta, X_{-\beta_i}]X_{\beta_i}+[X_\beta, X_{-(\beta-\beta_i)}]X_{\beta-\beta_i}=d_{\beta_i}X_\beta,\,\,\text{for some $d_{\beta_i} \in \Z_{>0}$.}
\end{equation}
 To prove this, write 
\begin{equation} \label{eqn12.5.4}
[X_\beta, X_{-\beta_i}]=cX_{\beta-\beta_i},\,\,\text{for some $c \neq 0$. }
\end{equation}
 Then, 
 \begin{equation*}
\begin{aligned}
([X_\beta, X_{-(\beta-\beta_i)}], X_{-\beta_i}) &=-c(X_{-(\beta-\beta_i)}, X_{\beta-\beta_i}), \ \text{by \eqref{eqn12.5.4}} \\
&=-c.
\end{aligned}
\end{equation*}
Thus, 
\begin{equation} \label{eqn12.5.5}
[X_\beta, X_{-(\beta-\beta_i)}] = -cX_{\beta_i}.
\end{equation}
Combining \eqref{eqn12.5.4} and \eqref{eqn12.5.5}, we get 
\begin{align*}
\hspace{0.5em} [X_{\beta}, X_{-\beta_i}] X_{\beta_i}+[X_\beta, X_{-(\beta-\beta_i)}]X_{\beta-\beta_i} &= c[X_{\beta-\beta_i}, X_{\beta_i}] \\
&= [[X_{\beta}, X_{-\beta_i}], X_{\beta_i}], \ \text{by \eqref{eqn12.5.4}} \\
&=d_{\beta_i}X_\beta,\,\, \text{for some $d_{\beta_i} \in \Z_{>0}$ by $sl_2$-module theory.}
\end{align*}
This proves \eqref{eqn12.5.6}. Substituting \eqref{eqn12.5.6} into \eqref{eqn12.5.7}, we get 
\begin{equation}
(l+m+h^\vee) X_\beta(k) \cdot v^\ast_\lambda(L_{-k}(w)) = ( lk+(\lambda_{|\mathring{\mf[h]}}, 
\beta_{|\mathring{\mf[h]}})+(\mu, \beta) +km + \frac{1}{2} \sum_{\substack{\beta_i \in \mathring{\Phi}^+:\\   \beta-\beta_i \in \mathring{\Phi}^+}} d_{\beta_i}) w_\mu \neq 0,
\end{equation}
since $\beta\in \mathring{\Phi}^+$. This proves the first part of the lemma. \\

Now, consider the case where $\beta=\delta-\gamma$ for some $\gamma \in \mathring{\Phi}^+$. Recall that $\theta$ is the highest root of $\mathring{\mf[g]}$. By equation \eqref{eqn12.5.1}, 
\begin{equation*}
\begin{aligned}
(l+m+h^\vee) X_{-\gamma}(k+1) \cdot v^\ast_\lambda(L_{-k}(w))& = l(k+1) X_{-\gamma}(1) \cdot w_{\mu-\beta}  - \sum_j \lambda(h_j) \gamma(h^j)X_{-\gamma}(1) \cdot w_{\mu-\beta} \\
 &- \sum_{\substack{\beta_i \in \mathring{\Phi}^+ :\\ \beta_i+\gamma \in \mathring{\Phi}^+}} v^\ast_\lambda(X_{\beta_i} \cdot v^\clubsuit_{\lambda-\beta_i})[X_{-\gamma}, X_{-\beta_i}](1) \cdot w^\clubsuit_{\mu-\beta+\beta_i} \\
 =& l(k+1) X_{-\gamma}(1) \cdot w_{\mu-\beta} - (\lambda_{|\mathring{\mf[h]}}, \gamma_{|\mathring{\mf[h]}})X_{-\gamma}(1) \cdot w_{\mu-\beta} \\
 &- \sum_{\substack{\beta_i \in \mathring{\Phi}^+:\\ \beta_i+\gamma \in \mathring{\Phi}^+}} v^\ast_\lambda(X_{\beta_i}X_{-\beta_i} \cdot  v_\lambda) [X_{-\gamma}, X_{-\beta_i}](1)  \cdot w^\circ_{\mu-\beta+\beta_i} \\
 =& l(k+1) X_{-\gamma}(1) \cdot w_{\mu-\beta} - (\lambda_{|\mathring{\mf[h]}}, \gamma_{|\mathring{\mf[h]}}) X_{-\gamma}(1) \cdot w_{\mu-\beta} \\
 &+ \sum_{\substack{\beta_i \in \mathring{\Phi}^+: \\ \beta_i+\gamma \in \mathring{\Phi}^+}} [X_{-\gamma}, X_{-\beta_i}](1) X_{\beta_i} \cdot w_{\mu-\beta},\,\,\text{using equation \eqref{eqn12.5.3}.}
 \end{aligned}
\end{equation*}
Finally, we get
\begin{align} \label{eqn12.6.10} 
(l+m+h^\vee)&X_{-\gamma}(k+1)\cdot v_\lambda^*(L_{-k}w)
=l(k+1)X_{-\gamma}(1)\cdot w_{\mu-\beta}\notag\\
&-(\lambda_{|\mathring{\h}}, \gamma_{|\mathring{\h}})X_{-\gamma}(1)\cdot w_{\mu-\beta}+\sum_{\substack{\beta_i \in \mathring{\Phi}^+:\\ \beta_i+\gamma \in \mathring{\Phi}^+}}[X_{-\gamma}(1), X_{-\beta_i}]X_{\beta_i}\cdot w_{\mu-\beta}.
\end{align}

\end{proof}

\begin{Prop} \label{prop12.2} Let $(\lambda, \mu, \beta)$ be a Wahl triple for a real root $\beta$. Then, the condition (b) of Proposition \ref{prop12.1} (specifically, the identity \eqref{eqn12.4new}) is satisfied for some embedding $V(\lambda+\mu-\beta) \subset V(\lambda) \otimes V(\mu)$ in the following cases:

 \vskip1ex
(a) $\mathring{\g}$ is arbitrary but $\beta$ is of the form $k\delta +\gamma $, for any $\gamma\in \mathring{\Phi}^+$ and $k\geq 0$.

\vskip1ex
(b) $\mathring{\g}$ is arbitrary but $\beta$ is of the form $(k+1)\delta - \theta$, for any $k\geq 0$, where $\theta$ is the highest root of $\mathring{\g}$. 

\vskip1ex
(c) $\mathring{\g}$ is simply-laced.

\vskip1ex
(d) $\mathring{\g}$ is of type $B_\ell \,(\ell \geq 2)$.

\vskip1ex
(e) $\mathring{\g}$ is of type $C_\ell \,(\ell \geq 2)$.

\vskip1ex
(f) $\mathring{\g}$ is not of type $G_2$ and $\lambda$ is regular dominant.
    \end{Prop}
\begin{proof} (a) By \cite{Ku1}, $\S$2.7, the proposition is true for $\beta =\gamma$. Thus, (a) follows follows from Lemma 
\ref{lem12.5}.

\vskip1ex
(b) In this case $\beta = \alpha_0+k\delta$ for $k\geq 0$. For $\beta =\alpha_0$, the proposition follows from the proof of Lemma \ref{lem9.1}. Now, to prove the proposition for $\beta = \alpha_0+k\delta$,  by the identity \eqref{eqn12.6.10}
of the proof of Lemma \ref{lem12.5}, we get that the right side of the identity \eqref{eqn12.6.10} equals $(lk+(\lambda, \alpha_0^\vee))w_\mu \neq 0$. This proves (b).

\vskip1ex
(c) Using (a) and (b), it suffices to prove the proposition for  $\beta = (k+1)\delta -\gamma$, for $k\geq 0$ and $\gamma\in \mathring{\Phi}^+\setminus \{\theta\}.$ In this case the proposition follows from the proof of Proposition 
\ref{oneroot}.

\vskip1ex
(d) As in (c), it suffices to prove the  proposition for  $\beta = (k+1)\delta -\gamma$, for $k\geq 0$ and $\gamma\in \mathring{\Phi}^+\setminus \{\theta\}.$ Observe first that if $\beta - 2\alpha_i\notin \Phi^+$ for any simple root $\alpha_i 
\, (0\leq i\leq \ell)$, then by the proof of Proposition \ref{oneroot}, the proposition follows. So, assume that 
$\beta - 2\alpha_{i_o}\in \Phi^+$  for some (and hence unique by Lemma \ref{root}) simple root $\alpha_{i_o}$. If 
$i_o\neq 0$, then by $\S$9.1 for $B_\ell^{(1)}, F_\beta = F_{\beta-\alpha_{i_o}}$. (Observe that $F_\beta =F_{\beta+k\delta}$ for any $k\in \mathbb{Z}$.)
Hence, the proposition in this case follows from the proof of Proposition \ref{tworoot}. So, assume now that $i_o=0$. In this case, by the proof of Lemma \ref{lem9.1}, $\beta = (k+1) \delta-\theta$, which is already covered by (b). This completes the proof of the proposition in the case (d). 

\vskip1ex
(e) As in (d), it suffices to prove the  proposition for  $\beta = (k+1)\delta -\gamma$, for $k\geq 0$ and $\gamma\in \mathring{\Phi}^+\setminus \{\theta\}$ such  that 
$\beta - 2\alpha_{j}\in \Phi^+$  for some (and hence unique) simple root $\alpha_{j}\, (1\leq j\leq \ell)$. Thus,   $\beta = (k+1)\delta -\gamma_i,\, k\geq 0$, where $\gamma=\gamma_i\,(2\leq i\leq \ell)$ is as in Example of $C_\ell^{(1)}$ in Section 9. Further, using Lemma \ref{Newlemma}, we can assume that $\lambda =\mu = \rho_\beta  =\Lambda_{i-1}$ (cf.  Example of $C_\ell^{(1)}$ in Section 9). (Observe that if the proposition is valid for the Wahl triple $(\lambda, \mu, \beta)$, then the proposition is valid for  $(\lambda+\lambda', \mu+ \mu', \beta)$ for any dominant weights $\lambda', \mu'\in \dom$ by the proof of Corollary \ref{additive}.) 

For $\beta=\delta-\gamma_i$, the proposition follows from the proof of Proposition \ref{oneroot} since $\Lambda_{i-1}(\beta^\vee)=1.$ We now prove the proposition for  $\beta=(k+1)\delta-\gamma_i$ for $k\geq 1$ (and  $\lambda= \mu=\Lambda_{i-1}$). We use the identity \eqref{eqn12.6.10} of the proof of Lemma \ref{lem12.5} to show that in this case 
$X_{-\gamma_i}(k+1)\cdot v_\lambda^*(L_{-k}w) \neq 0:$

By the proof of Proposition \ref{oneroot}, we can take $w_{\Lambda_{i-1}-(\delta-\gamma_i)} = X_{\gamma_i}(-1)\cdot w_{\Lambda_{i-1}}$. We freely use the notation from \cite{Bou}, Planche III. Then, for any $\beta_j \in \mathring{\Phi}^+$ such that  $\beta_j+\gamma_i \in \mathring{\Phi}^+$, we get that $\beta_j$ is precisely of the form $\epsilon_j-\epsilon_i$ for $1\leq j <i$ 
(and $\gamma_i =2\epsilon_i$). Thus,
\begin{align} [X_{-\gamma_i}(1), X_{-\beta_j}]X_{\beta_j} X_{\gamma_i}(-1)\cdot w_{\Lambda_{i-1}}&= 
\left[[X_{-\gamma_i}(1), X_{-\beta_j}], [X_{\beta_j} ,X_{\gamma_i}(-1)]\right]\cdot w_{\Lambda_{i-1}}\notag\\
&= \left([X_{-\gamma_i}(1), X_{-\beta_j}], [X_{\beta_j} ,X_{\gamma_i}(-1)]\right)(\Lambda_{i-1}, \delta-\gamma_i-\beta_j) w_{\Lambda_{i-1}},\notag\\
&\medskip\,\,\,\,\,\,\,\,\,\,\,\,\,\,\text{ by \cite{Ku3}, Theorem 1.5.4}\notag\\
&= \left(X_{-\gamma_i}(1), [X_{-\beta_j}, [X_{\beta_j} ,X_{\gamma_i}(-1)]]\right)(\Lambda_{i-1}, \delta-\gamma_i-\beta_j) w_{\Lambda_{i-1}}\notag\\
&= \left(X_{-\gamma_i}(1), 2(\gamma_i(\beta_j^\vee)+3)X_{\gamma_i}(-1)\right)(\Lambda_{i-1}, \delta-\gamma_i-\beta_j) w_{\Lambda_{i-1}},\notag\\
&\text{since $(\ad X_{-\beta_j})X_{\gamma_i}=(\ad X_{\beta_j})^3X_{\gamma_i}=0$ and $(\ad X_{\beta_j})^2X_{\gamma_i}\neq 0$}\notag\\
&=w_{\Lambda_{i-1}},\,\,\,\text{since $(\Lambda_{i-1}, \delta-\gamma_i-\beta_j)= 1/2$ and $\gamma_i(\beta_j^\vee)=2$}.
\end{align}
Hence,
$$\sum_{1\leq j <i} \,[X_{-\gamma_i}(1), X_{-\beta_j}]X_{\beta_j} X_{\gamma_i}(-1)\cdot w_{\Lambda_{i-1}}= (i-1)w_{\Lambda_{i-1}}.$$
Further,
$$X_{-\gamma_i}(1)X_{\gamma_i}(-1)\cdot w_{\Lambda_{i-1}}= [X_{-\gamma_i}(1), X_{\gamma_i}(-1)]\cdot w_{\Lambda_{i-1}}= w_{\Lambda_{i-1}},\,\,\text{since $l=m:=\Lambda_{i-1}(K)=1$}.$$
So, the right side of  the identity \eqref{eqn12.6.10} of the proof of Lemma \ref{lem12.5} becomes 
$$\left((k+1) +i-1\right)w_{\Lambda_{i-1}}= (k+i)w_{\Lambda_{i-1}}\neq 0.$$
This proves the proposition in this case.

\vskip1ex
(f) Proof of the proposition in this case follows from the proofs of Propositions \ref{oneroot} and \ref{tworoot}. 

\end{proof}

   {\rm Following is the main theorem of this section:}
   
    \begin{Thm} \label{thm12.1} Let $\g$ be an affine Kac-Moody Lie algebra and let  $w\in W_\cP'$ be such that  the Schubert variety $X_w^\cP$ is $\cP$-stable. Then, for any   $\lambda, \mu \in  {\dom_S}^o $ (where $S$ is an arbitrary subset of the simple roots) such that the condition (b) of Proposition  \ref{prop12.1} is satisfied for all the Wahl triples $(\lambda, \mu, \beta)$ for any real root $\beta \in \Phi^+$ (cf. Proposition 
 \ref{prop12.2}),
   \begin{equation} \label{eqn12.6.1} H^p(\hat{\X}_w^\cP, \hat{\cI}_e^2\otimes \cL_w(\lambda\boxtimes \mu)) = 0,\,\,\,\text{for all $p>0$}.
   \end{equation}
   In particular, the canonical Gaussian map 
    \begin{equation} \label{eqn12.6.2}    
   H^0(\hat{\X}_w^\cP, \hat{\cI}_e\otimes \cL_w(\lambda\boxtimes \mu))   \to  H^0(\hat{\X}_w^\cP, (\hat{\cI}_e/\hat{\cI}_e^2)\otimes \cL_w(\lambda\boxtimes \mu)) \end{equation}
   is surjective.  
   
   In particular, \eqref{eqn12.6.1} and \eqref{eqn12.6.2} are valid for any simply-laced $\mathring{\g}$ and $\mathring{\g}$ of types $B_\ell, C_\ell$. Moreover, they also are valid for   $\mathring{\g}$ of type $F_4$ in the case $\cP$ is the Borel subgroup $\B$. 
   \end{Thm}
   \begin{proof} By Proposition \ref{prop12.1}, we get the vanishing 
   $H^1(\hat{\X}_w^\cP, \hat{\cI}_e^2\otimes \cL_w(\lambda\boxtimes \mu)) = 0.$
    
    To prove  the higher cohomology vanishing, consider the cohomology exact sequence, associated to the sheaf exact sequence over $\hat{\X}_W^\cP$:
\begin{equation} \label{eqn12.6} 0\to \hat{\cI}_e^2\otimes \cL_w(\lambda\boxtimes \mu)\to\cL_w(\lambda\boxtimes \mu)\to  (\hat{\cO}_w/\hat{\cI}_e^2)\otimes \cL_w(\lambda\boxtimes \mu) \to 0.
\end{equation}

First of all by \cite{Ku2}, Theorem 2.7 (also see \cite{M2}, Corollaire 3),
\begin{equation}\label{eqn23}
H^p( \hat{\X}^\cP_w,  \cL_w(\lambda\boxtimes \mu)) =0,\,\,\,\text{for all $p>0$}
\end{equation}
Further, by Lemma \ref{lem12.1} and Proposition \ref{prop11.1},
\begin{equation}\label{eqn24}H^p( \hat{\X}^\cP_w,  (\hat{\cO}_w/\hat{\cI}_e^2)\otimes \cL_w(\lambda\boxtimes \mu)) \simeq H^p(\X_\cP, \cL(\lambda) \otimes \cL((\hat{T}_e(X_w^\cP)\cdot v_\mu)^*)).
\end{equation}
Now,  we prove the vanishing: 
\begin{equation}\label{eqn25}H^p({\X}_\cP,  \cL(\lambda) \otimes \cL((\hat{T}_e(X_w^\cP)\cdot v_\mu)^*)) =0,\,\,\,\text{for all $p>0$}:
\end{equation}
Take a $\cP$-module filtration of the finite dimensional $\cP$-module $M_w:=  \C_{\lambda}^*\otimes 
(\hat{T}_e(X_w^\cP)\cdot v_\mu)^*$:
$$0= F_0\subset F_1 \subset \cdots \subset F_d=M_w,$$
such that $V_j:= F_j/F_{j-1}$ is an irreducible $\cP$-module for each $1\leq j\leq d$. By Engel's theorem, $V_j$ is a trivial $\cU$-module, thus it can be considered as an irreducible module for the (finite dimensional) Levi algebra $\mathfrak{l}:= \p/\fu$. Let $\delta_j$ be the lowest weight of $V_j$ as an $\mathfrak{l}$-module. Thus,
 \begin{equation}\label{eqn26} \delta_j(\alpha_k^\vee) \leq 0,\,\,\,\text{for all $k\in S$}.
 \end{equation}
To prove the vanishing \eqref{eqn25}, it suffices to show that 
\begin{equation}\label{eqn27} H^p(\X_\cP, \cL(V_j)) = 0,\,\,\,\text{for all $p>0$}.
\end{equation}
From the classical Borel-Weil-Bott theorem for the finite dimensional reductive Lie algebra $\mathfrak{l}$ and the Leray spectral sequence for the fibration $\G/\B \to \G/\cP$, we get 
\begin{equation}\label{eqn28}
H^p(\X_\cP, \cL(V_j)) \simeq H^p(\X_\B, \cL(-\delta_j)), \,\,\,\text{for all $p\geq 0$}.
\end{equation}
Observe further that $\delta_j$, in particular, is a weight of   $\C_{\lambda}^*\otimes 
(\hat{T}_e(X_w^\cP)\cdot v_\mu)^*$ which is a quotient of  $\C_{\lambda}^*\otimes 
(\g\cdot v_\mu)^\vee$. Thus, $\delta_j$ is of the form $\delta_j= -(\lambda +\mu-\beta)$, for some $\beta\in \Phi^+(S)\cup \{0\}$. Hence, for $i\notin S$, since $\lambda, \mu\in {\dom_S}^o $, by \cite{Bou}, Page 278, item 6,
  \begin{equation}\label{eqn29} -\delta_j(\alpha_i^\vee) \geq -1.
 \end{equation}
Thus, combining the equations \eqref{eqn26},  \eqref{eqn29} and   \eqref{eqn28}, the vanishing  \eqref{eqn27}
follows from the affine analogue of the Borel-Weil-Bott theorem (cf. \cite{Ku3}, Corollary 8.3.12). This proves the vanishing \eqref{eqn25}.

Thus, the cohomology exact sequence associated to the sheaf exact sequence \eqref{eqn12.6} and the above two vanishing results \eqref{eqn23} and \eqref{eqn25} (using the identification \eqref{eqn24}) prove \eqref{eqn12.6.1} and \eqref{eqn12.6.2} of the theorem.

The last `In particular' statement of the theorem follows from Proposition \ref{prop12.2}. 
 \end{proof}
    
    {\rm For any $w\in W_\cP'$ such that $X_w^\cP$ is $\cP$-stable, we think of $\hat{\X}_w^\cP$ as a sub ind-variety of   
$\X_\cP\times  \X_\cP$ via the isomorphism $\delta: \G\times^\cP \X_\cP\to \X_\cP\times  \X_\cP$ (as in the beginning of Section 12). From now on, we think of the sheaf $\hat{\cI}_e^k = \hat{\cI}_e(w)^k$ over $\hat{\X}_w^\cP$ (for any $k\geq 1$) as a sheaf over 
$\X_\cP\times  \X_\cP$ by considering $\delta_*(i_w)_*(\hat{\cI}_e(w)^k)$, where $i_w: 
 \G\times^\cP X_w^\cP \to   \G\times^\cP \X_\cP$ is the canonical embedding. Let $\cI_D$ be the ideal sheaf of the diagonal $D \subset \X_\cP\times  \X_\cP$. Then, from the sheaf exact sequence:
 $$0\to \hat{\cI}_e(w)\to \hat{\cO}_w\to \hat{\cO}_e\to 0,$$
  we get that 
 $$ \cI_D \simeq \varprojlim_w\,\hat{\cI}_e(w).$$
 Define 
 $$ \tilde{\cI}_D^2 := \varprojlim_w\,\hat{\cI}_e(w)^2.$$ 
 Observe that $\cI_D^2 \subset  \tilde{\cI}_D^2$ and the inclusion is strict in general.

As a corollary of the above  theorem, we get the following. This was conjectured in the finite case by Wahl \cite{Wah} and proved in that case by Kumar \cite{Ku1}. }
 
 \begin{Cor} \label{coro12.7} Under the notation and assumptions of Theorem \ref{thm12.1}, the canonical Gaussian map 
 $$H^0(\X_\cP\times \X_\cP, \cI_D\otimes \cL(\lambda\boxtimes \mu))  \to H^0(\X_\cP\times \X_\cP, (\cI_D/\tilde{\cI}_D^2)\otimes \cL(\lambda\boxtimes \mu))$$
 is surjective.  
 
 In particular, it is surjective for any simply-laced $\mathring{\g}$ as well as $\mathring{\g}$ of types $B_\ell, C_\ell$. In addition, it is surjective for $\mathring{\g}$ of type $F_4$ in the case $\cP$ is the Borel subgroup $\B$. 
 \end{Cor}
 \begin{proof} By Theorem \ref{thm12.1}, for any $w\in W_\cP'$ such that $X_w^\cP$ is $\cP$-stable, the canonical map $$\eta_w: H^0(\X_\cP\times \X_\cP, \hat{\cI}_e(w)\otimes \cL(\lambda\boxtimes \mu))   \to  H^0(\X_\cP \times \X_\cP, (\hat{\cI}_e(w)/\hat{\cI}_e(w)^2)\otimes \cL(\lambda\boxtimes \mu)) $$
   is surjective.  Moreover, $\eta_w$ is a $\G$-module morphism between $\G$-modules which are finite direct sums of $V(\theta)^*$ for some $\theta \in \dom$ (cf. proof of Proposition \ref{prop12.1}). Let $K_w$ be the kernel of $\eta_w$. Then, $K_w$ being a finite direct sum of $V(\theta)^*$, it is easy to see that $\{K_w\}_w$ satisfies the Mittag-Leffler condition (cf. \cite{H2}, Chap. II, $\S$9). Thus, by  \cite{H2}, Chap. II, Proposition 9.1, the canonical map
   \begin{equation} \label{eqn30}
   \varprojlim_w\, H^0(\X_\cP\times \X_\cP, \hat{\cI}_e(w)\otimes \cL(\lambda\boxtimes \mu))   \to    \varprojlim_w\   H^0(\X_\cP \times \X_\cP, (\hat{\cI}_e(w)/\hat{\cI}_e(w)^2)\otimes \cL(\lambda\boxtimes \mu))
   \end{equation} is surjective.   Further, by  \cite{H2}, Chap. II, Proposition 9.2,
    \begin{equation} \label{eqn31}
   \varprojlim_w\, H^0(\X_\cP\times \X_\cP, \hat{\cI}_e(w)\otimes \cL(\lambda\boxtimes \mu))  \simeq H^0(\X_\cP\times \X_\cP, \cI_D\otimes \cL(\lambda \boxtimes \mu))
   \end{equation}
   and
    \begin{equation} \label{eqn32}   
     \varprojlim_w\   H^0(\X_\cP \times \X_\cP, (\hat{\cI}_e(w)/\hat{\cI}_e(w)^2)\otimes \cL(\lambda\boxtimes \mu))\simeq H^0(\X_\cP\times \X_\cP, (\varprojlim_w\, \hat{\cI}_e(w)/\hat{\cI}_e(w)^2)\otimes \cL(\lambda \boxtimes \mu)).    
   \end{equation} 
 From the left exactness of $ \varprojlim$ in the category of sheaves, 
 \begin{equation} \label{eqn33}   
   j: \cI_D/\tilde{\cI}_D^2 \hookrightarrow  
    \varprojlim_w\, \hat{\cI}_e(w)/\hat{\cI}_e(w)^2.  
   \end{equation}  
 Hence, by combining the equations \eqref{eqn30}-\eqref{eqn33}, we get that the composite
 \begin{align*} H^0(\X_\cP\times \X_\cP,\cI_D\otimes & \cL(\lambda\boxtimes \mu))\overset{\eta}{\to} H^0(\X_\cP\times \X_\cP, 
 \cI_D/\tilde{\cI}_D^2 \otimes \cL(\lambda\boxtimes \mu)) \\ &
 \overset{j_*}{\hookrightarrow}
  H^0(\X_\cP\times \X_\cP,    (\varprojlim_w\, \hat{\cI}_e(w)/\hat{\cI}_e(w)^2) \otimes  \cL(\lambda\boxtimes \mu)) 
  \end{align*}
   is surjective and hence so is $\eta$ (and $j_*$ is an isomorphism). This proves the corollary.  
 \end{proof}
 \begin{Rmk} {\rm (a) It is very likely that, under the assumptions of Corollary \ref{coro12.7},  $H^p(\X_\cP\times \X_\cP, \tilde{\cI}_D^2\otimes \cL(\lambda\boxtimes \mu)) =0$ for all $p>0$ (cf. \cite{H1},   Chap. I, Theorem 4.5). Moreover, for $p=1$, it is equivalent to the condition (b) of Proposition \ref{prop12.1}.
 \vskip1ex
 
 (b) It is likely that Theorem \ref{thm12.1} (and hence Corollary \ref{coro12.7}) is valid for any $\g$ and any $\lambda, \mu  \in  {\dom_S}^o. $  }
 \end{Rmk}


\begin{thebibliography}{KlRa}


\bibitem[BJK]{BJK} Besson, M., Jeralds, S., Kiers, J., Multiplicity in root components via geometric Satake, arXiv: 1909.05103. 


\bibitem[Bou]{Bou}Bourbaki, N., Lie groups and Lie algebras, Chapters 4-6, Springer (2005).

\bibitem[BrKu]{BrKu} Brown, M., Kumar, S., A study of saturated tensor cone for symmetrizable Kac--Moody algebras, Math. Ann. \textbf{360} (2014). 

\bibitem[H1]{H1} Hartshorne, R., On the De Rham cohomology of algebraic varieties, Publ. Math. IHES \textbf{45} (1976).

\bibitem[H2]{H2} Hartshorne, R., Algebraic Geometry, Springer-Verlag (1977).

\bibitem[Kac]{Kac} Kac, V., Infinite dimensional Lie algebras, Cambridge University Press (1990).

\bibitem[KRR]{KRR} Kac, V., Raina, A., Rozhkovskaya, N., Bombay lectures on highest weight representations of infinite dimensional Lie algebras, Adv. Series in Mathematical Physics, vol. 29, World Scientific (2013). 

\bibitem[Kas1]{Kas1} Kashiwara, M., On crystal bases of the Q-analogue of universal enveloping algebras, Duke Math Journal \textbf{63} (1991). 

\bibitem[Kas2]{Kas2} Kashiwara, M., Global bases of quantum groups, Duke Math Journal \textbf{69} (1993). 

\bibitem[Kos]{Kos} Kostant, B., A formula for the multiplicity of a weight, Trans. Am. Math. Soc. \textbf{93} (1959). 

\bibitem[Ku1]{Ku1} Kumar, S., Proof of Wahl's conjecture on surjectivity of the Gaussian map for flag varieities, American Journal of Mathematics \textbf{114} (1992).

\bibitem[Ku2]{Ku2} Kumar, S., Existence of certain components in the tensor product of two integrable highest weight modules for Kac--Moody algebras, in: ``Infinite dimensional Lie algebras and groups" (ed. by V.G. Kac). Advanced series in Mathematical Physics Vol. 7, World Scientific (1989). 

 \bibitem[Ku3]{Ku3}Kumar, S., Kac--Moody groups, their flag varieties, and representation theory, Birkh\"auser (2002).
 
 \bibitem[Ku4]{Ku4} Kumar, S., Tensor product decomposition, Proceedings of the International Congress of Mathematician, Hyderabad, India (2010).
 \bibitem[M1]{M1} Mathieu, O., Formules de caract\`eres pour les alg\`ebres de Kac-Moody g\'en\'erales, 
 Asterisque vol. \textbf{159-160} (1988).
  \bibitem[M2]{M2} Mathieu, O., Construction d'un groupe de Kac-Moody et applications, Compositio Math. \textbf{69} (1989). 
  
 \bibitem[Wah]{Wah} Wahl, J., Gaussian maps and tensor products of irreducible representations, Manuscripta Math. \textbf{73} (1991).


\end{thebibliography}
\end{document}